\documentclass[a4paper]{amsart}
\usepackage{graphicx,amsmath,amsfonts,latexsym,amssymb,amsthm,mathrsfs, color,hyperref}
\usepackage[latin1]{inputenc}
\usepackage{amsthm}
\usepackage{amssymb}
\theoremstyle{plain}
\newtheorem{thm}{Theorem}[section]
\newtheorem{proposition}[thm]{Proposition}
\newtheorem{corollary}[thm]{Corollary}
\newtheorem{definition}[thm]{Definition}
\newtheorem{theorem}[thm]{Theorem}
\newtheorem{lemma}[thm]{Lemma}
\newtheorem{remark}[thm]{Remark}
\newtheorem*{domainL*}{Domain of Laplacian}
\newtheorem*{domain*}{Domain of bi-Laplacian}

\begin{document}
\title[The Cahn-Hilliard equation on manifolds with conical singularities]{Smoothness and long time existence for solutions of the Cahn-Hilliard equation on manifolds with conical singularities}

\author{Pedro T. P. Lopes}
\address{Instituto de Matem\'atica e Estat\'istica, Universidade de S\~ao Paulo, Rua do Mat\~ao 1010, 05508-090, S\~ao Paulo, SP, Brazil}
\email{pplopes@ime.usp.br}
\thanks{Pedro T. P. Lopes was partially supported by FAPESP 2016/07016-8 and FAPESP 2019/15200-1}

\author{Nikolaos Roidos}
\address{Department of Mathematics, University of Patras, 26504 Rio Patras, Greece}
\email{roidos@math.upatras.gr}
\thanks{Nikolaos Roidos was supported by Deutsche Forschungsgemeinschaft, grant SCHR 319/9-1}

\subjclass[2010]{35K58; 35K65; 35K90; 35K91; 35R01}

\date{\today}
\begin{abstract} 
We consider the Cahn-Hilliard equation on manifolds with conical singularities. For appropriate initial data, we show that the solution exists in the maximal $L^q$-regularity space for all times and becomes instantaneously smooth in space and time, where the maximal $L^q$-regularity is obtained in the sense of Mellin-Sobolev spaces. Moreover, we provide precise information concerning the asymptotic behavior of the solution close to the conical tips in terms of the local geometry.
\end{abstract}

\keywords{Semilinear parabolic equations, maximal regularity, Cahn-Hilliard equation, manifolds with conical singularities.}

\maketitle

\section{Introduction}

Let $\mathcal{B}$ be a smooth compact $(n+1)$-dimensional manifold, $n\geq 1$, with closed (i.e. compact without boundary) possibly disconnected smooth boundary $\partial\mathcal{B}$ of dimension $n$. We endow $\mathcal{B}$ with a degenerate Riemannian metric $g$ which in a collar neighborhood $[0,1)\times\partial\mathcal{B}$ of the boundary is of the form
$$
g=dx^{2}+x^2h(x),
$$
where $h(x)$, $x\in[0,1)$, is a family of Riemannian metrics on $\partial\mathcal{B}$ that is smooth and does not degenerate up to $x=0$. We call $\mathbb{B}=(\mathcal{B},g)$ a {\em manifold with conical singularities} or {\em conic manifold}; the singularities, i.e. the conical tips, correspond to the boundary $\{0\}\times \partial\mathcal{B}$ of $\mathcal{B}$. 

We consider the following semilinear equation on $\mathbb{B}$, namely 
\begin{eqnarray}\label{CH1}
u'(t)+\Delta^{2}u(t)&=&\Delta(u^{3}(t)-u(t)), \quad t\in(0,T),\\\label{CH2}
u(0)&=&u_{0},
\end{eqnarray}
where $\Delta=\Delta_{g}$ is the Laplacian on $\mathbb{B}$, $T>0$ and $u_{0}$ is an appropriate initial data. The above problem, called {\em the Cahn-Hilliard equation}, is a phase-field or diffuse interface equation which models the phase separation of a binary mixture. The scalar function $u$ corresponds to the difference of the components; the sets where $u=\pm 1$ correspond to domains of pure phases. For global existence and the properties of the solutions on smooth domains we refer to \cite{CaMu}, \cite{CaHi}, \cite{EZ}, \cite{Temam} and \cite{vonwahl}.

When the Laplacian $\Delta$ on $\mathbb{B}$ is restricted to the collar part $(0,1)\times\partial\mathcal{B}$ it takes the degenerate form
$$
\Delta=\frac{1}{x^{2}}\Big((x\partial_{x})^{2}+(n-1+\frac{x\partial_{x}\det[h(x)]}{2\det[h(x)]})(x\partial_{x})+\Delta_{h(x)}\Big),
$$
where $\Delta_{h(x)}$ is the Laplacian on $\partial\mathcal{B}$ induced by the metric $h(x)$. We regard $\Delta$ as an element in the class of cone differential operators or Fuchs type operators, i.e. the naturally appearing degenerate differential operators on $\mathbb{B}$. These operators act naturally on the scales of weighted Mellin-Sobolev spaces $\mathcal{H}_{p}^{s,\gamma}(\mathbb{B})$, $p\in(1,\infty)$, $s,\gamma\in \mathbb{R}$, see Definition \ref{mellinsobolevspaces}. Due to the degeneracy, when $\Delta$ is considered as an unbounded operator in $\mathcal{H}_{p}^{s,\gamma}(\mathbb{B})$ it admits several closed extensions; all of them differ by a finite dimensional space that is determined explicitly by the family $h(\cdot)$, see Section \ref{realizations} for details.

We focus on a non-trivial closed extension of $\Delta$ having domain $\mathcal{H}_{p}^{s+2,\gamma+2}(\mathbb{B})\oplus\mathbb{C}_{\omega}$, where $\mathbb{C}_{\omega}$ is a certain subspace of the space of smooth functions on $\mathbb{B}$; it consists of functions that are locally constant close to the singularities, see Definition \ref{constfunt}. In this situation, the domain of the associated bi-Laplacian is as in \eqref{bilapintro}, where $\mathcal{E}_{\Delta^2,\gamma}$ is a finite dimensional $s$-independent space consisting of linear combinations of functions that are smooth in the interior of $\mathbb{B}$ and, in local coordinates $(x,y)\in (0,1)\times\partial\mathcal{B}$ on the collar part, they are of the form $c(y)x^{-\rho}\log^{k}(x)$, where $c\in C^{\infty}(\partial\mathbb{B})$, $\partial\mathbb{B}=(\partial\mathcal{B},h(0))$, $\rho\in\{z\in\mathbb{C}\, |\, \mathrm{Re}(z)\in [\frac{n-7}{2}-\gamma,\frac{n-3}{2}-\gamma)\}$ and $k\in\{0,1,2,3\}$. The powers $\rho$, $k$, and in general the subspace $\mathcal{E}_{\Delta^2,\gamma}$, are determined explicitly by the family $h(\cdot)$ and the weight $\gamma$.

 Based on bounded imaginary powers results for the above described Laplacian, together with the theory of maximal $L^{q}$-regularity, we show the following result concerning \eqref{CH1}-\eqref{CH2}. For $\xi\in(0,1)$ and $\eta\in(1,\infty)$, denote by $(\cdot,\cdot)_{\xi,\eta}$ the real interpolation functor of type $\xi$ and exponent $\eta$.
 
\begin{thm}\label{ThLTS} Let $p\in(1,\infty)$, $s\ge0$, $s+2>\frac{n+1}{p}$,
\begin{gather}\label{gamma}
\frac{n-3}{2}<\gamma<\min\Big\{-1+\sqrt{\Big(\frac{n-1}{2}\Big)^{2}-\lambda_{1}} ,\frac{n+1}2\Big\} \quad \text{and}\quad \Big\{\begin{array}{lll} q>2 &\text{if}& p\ne2\\ q\geq2 &\text{if} & p=2 ,\end{array}
\end{gather}
where $\lambda_{1}$ is the greatest non-zero eigenvalue of the boundary Laplacian $\Delta_{h(0)}$ on $(\partial\mathcal{B},h(0))$. Moreover, denote by $\mathbb{C}_{\omega}$ the space of smooth functions on $\mathbb{B}$ that are locally constant close to the singularities, see Definition \ref{constfunt}. Then, for each
\begin{gather}\label{u0reg}
u_{0}\in(\mathcal{D}(\underline{\Delta}_{s}^{2}),\mathcal{H}_{p}^{s,\gamma}(\mathbb{B}))_{\frac{1}{q},q}\hookleftarrow \bigcup_{\varepsilon>0}\mathcal{H}_{p}^{s+4-\frac{4}{q}+\varepsilon,\gamma+4-\frac{4}{q}+\varepsilon}(\mathbb{B})\oplus\mathbb{C}_{\omega}
\end{gather}
there exists a $T>0$ and a unique 
\begin{gather}\label{uexists}
u \in W^{1,q}(0,T;\mathcal{H}_{p}^{s,\gamma}(\mathbb{B}))\cap L^{q}(0,T;\mathcal{D}(\underline{\Delta}_{s}^{2}))
\end{gather}
solving \eqref{CH1}-\eqref{CH2} on $[0,T]\times\mathbb{B}$, where the bi-Laplacian domain $ \mathcal{D}(\underline{\Delta}_{s}^{2})$ satisfies
\begin{gather}\label{bilapintro}
\mathcal{H}_{p}^{s+4,\gamma+4}(\mathbb{B})\oplus\mathbb{C}_{\omega}\oplus\mathcal{E}_{\Delta^{2},\gamma} \hookrightarrow \mathcal{D}(\underline{\Delta}_{s}^{2})\hookrightarrow \bigcap_{\varepsilon>0}\mathcal{H}_{p}^{s+4,\gamma+4-\varepsilon}(\mathbb{B})\oplus\mathbb{C}_{\omega}\oplus\mathcal{E}_{\Delta^{2},\gamma}
\end{gather}
for certain $s$-independent finite dimensional space $\mathcal{E}_{\Delta^{2},\gamma}${\em;} both $\mathcal{D}(\underline{\Delta}_{s}^{2})$ and $\mathcal{E}_{\Delta^{2},\gamma}$ described in Section \ref{realizations}, are determined by the family of metrics $h(\cdot)$ and the weight $\gamma$. Furthermore
\begin{gather}\label{extrareg}
u\in \bigcap_{s\ge0}C^{\infty}((0,T);\mathcal{D}(\underline{\Delta}_{s}^{2}))\hookrightarrow \bigcap_{\varepsilon>0} C^{\infty}((0,T);\mathcal{H}_{p}^{\infty,\gamma+4-\varepsilon}(\mathbb{B})\oplus\mathbb{C}_{\omega}\oplus\mathcal{E}_{\Delta^{2},\gamma})
\end{gather}
and
\begin{equation}\label{contsolu}
u \in \bigg\{\begin{array}{lll}
\bigcap_{\varepsilon>0}C([0,T);\mathcal{H}_{p}^{s+4-\frac{4}{q}-\varepsilon,\gamma+2+\delta_{0}(1-\frac{2}{q})}(\mathbb{B})\oplus\mathbb{C}_{\omega}) &\text{if} & q>2\\
C([0,T);\mathcal{H}_{2}^{s+2,\gamma+2}(\mathbb{B})\oplus\mathbb{C}_{\omega}) & \text{if} & p=q=2
\end{array}\bigg\} \hookrightarrow C([0,T);C(\mathbb{B})),
\end{equation}
where $\delta_{0}>0$ is fixed and is determined by $h(\cdot)$ and $\gamma$. In particular, when $h(\cdot)=h$ is constant, then $\mathcal{D}(\underline{\Delta}_{s}^{2})$, $\mathcal{E}_{\Delta^{2},\gamma}$ and $\delta_{0}$ are determined by $\gamma$ and the spectrum $\sigma(\Delta_{h})$.\\
If besides the above assumptions we assume that $n\in\{1,2\}$, $\gamma<-\frac{1}{2}$ for $n=1$ and $\gamma<-\frac{1}{4}$ for $n=2$, then the above $T>0$ can be taken arbitrary large, that is, the solutions are globally defined. 
\end{thm}

We note that the weight $\gamma$ plays an important role in the proof of global solutions. The fact that it has to be smaller than some constant, $-\frac{1}{2}$ or $-\frac{1}{4}$, is crucial for establishing a priori estimates for the solution, as we will see in the study of the energy functional. 

{\em Space asymptotics}. The above theorem provides information concerning the asymptotic behavior of the solutions close to the singularities as well as its relation with the local geometry. More precisely, from \eqref{extrareg} we can uniquely decompose the solution into three components, namely $u=u_{\mathbb{C}}\oplus u_{\mathcal{E}} \oplus u_{\mathcal{H}}$, where $u_{\mathbb{C}}\in C^{\infty}(0,T;\mathbb{C}_{\omega})$, $u_{\mathcal{E}}\in C^{\infty}(0,T;\mathcal{E}_{\Delta^{2},\gamma})$ and $ u_{\mathcal{H}}\in C^{\infty}(0,T;\mathcal{H}_{p}^{\infty,\gamma+4-\varepsilon})$ for all $\varepsilon>0$. Elements in $C^{\infty}(0,T;\mathbb{C}_{\omega})$ have constant values close to the singularities. Elements in $C^{\infty}(0,T;\mathcal{E}_{\Delta^{2},\gamma})$ admit an asymptotic expansion with respect to the geodesic distance $x$ from the conical tips in terms of real powers of $x$ and integer powers of $\log(x)$; these powers are time independent and are determined explicitly by $h
 (\cdot)$ and $\gamma$, see Section \ref{realizations}. Finally, due to standard embedding properties of Mellin-Sobolev spaces, see e.g. Lemma \ref{propms} (iii), $u_{\mathcal{H}}\in C^{\infty}(0,T;C(\mathbb{B}))$ and, in local coordinates $(x,y)\in [0,1)\times\partial\mathcal{B}$ on the collar part, for each $\beta<\gamma+4-\frac{n+1}{2}$, we have 
$$
|u_{\mathcal{H}}(t,x,y)|\leq Cx^{\beta}, \quad t\in (0,T),
$$
where the constant $C>0$ is determined explicitly by $u_{0}$, $\gamma$, $\beta$ and $t$. Therefore, a picture of how the geometry is reflected to the evolution is provided. 

{\em Related work}. In \cite{RS1} the problem \eqref{CH1}-\eqref{CH2} was studied on manifolds with straight conical tips and it was shown existence, uniqueness and maximal $L^{q}$-regularity for solutions on weighted $L^{p}$-spaces. Later on, in \cite{RS2} this result was extended to possibly warped cones and to higher order Mellin-Sobolev spaces. In both cases the solutions were obtained for short times and their provided regularity was dependent on the regularity of the initial data. Moreover, in \cite{Ve} the problem \eqref{CH1}-\eqref{CH2} was considered on manifolds with edge type singularities and short time results where shown in terms of incomplete edge-H\"older spaces.

In the present work, we show optimal results by generalizing the results of \cite{RS2} and \cite{RS1} from the following two aspects: {\em 1st} from the existence point of view by showing that the solutions are global in time when $\dim(\mathbb{B})=2$ or $3$ and {\em 2nd} from the regularity point of view by showing instantaneously smoothing for solutions in space and time. Moreover, we provide sharper space asymptotics of the solutions through the domain of the bi-Laplacian, which indicates a stronger influence of the underline local geometry of the conical tips to the evolution.

The starting point in our consideration is a gradient estimate obtained by the energy functional for the Cahn-Hilliard equation. Next, a combination of the moment inequality and the singular Gr\"onwall's inequality provides a uniform bound of the short time solution with respect to the norm associated with the domain of certain fractional power of the Laplacian. Then, we apply standard long time existence theory for maximal $L^q$-regular solutions of semilinear equations. Smoothness in space and time is obtained by a Banach scale increasing regularity argument that relies on a theorem of Pr\"uss and Simonett, based on Angenent \cite{Angenent}. We extensively use results obtained by means of the cone pseudodifferential calculus (see e.g. \cite{Le} and \cite{Schu}), such as the description of the domains of $\mathbb{B}$-elliptic cone differential operators (see \cite{GKM}, \cite{SS} and \cite{Sei}), and the boundedness of the imaginary powers of our underline Laplacian (see \cite{CSS2}, \cite{Lo}, \cite{Ro2} and \cite{SS1}).

\subsection*{Notation}

In this paper, whenever $E_{0}$ and $E_{1}$ are Banach spaces, we denote by $\mathcal{L}(E_{0},E_{1})$ the set of all bounded operators from $E_{0}$ to $E_{1}$ and write $\mathcal{L}(E_{0})=\mathcal{L}(E_{0},E_{1})$ if $E_{0}=E_{1}$. Similarly, we denote by $\mathcal{L}_{s}(E_{0},E_{1})$ the space $\mathcal{L}(E_{0},E_{1})$ equipped with strong operator topology. We also denote by $E_{0}'$ the dual space of $E_{0}$. All Banach spaces are complex. However, whenever we are considering a solution of \eqref{CH1}-\eqref{CH2}, we assume that the solution is real valued, which is equivalent to say that the initial condition $u_{0}$ is real valued. This is not so relevant in order to obtain short time solutions, but it is important for the study of the energy functional and of global existence of solutions. The set $\mathbb{N}$ always indicates the non-negative integers $\{0,1,2,...\}$ and the symbol $\oplus$ is used for the sum of Banach spaces, which is not always direct.

\section{Maximal $L^{q}$-regularity theory for semilinear parabolic problems}

Our main tool for studying the Cahn-Hilliard equation on manifolds with conical singularities is the theory of maximal $L^q$-regularity. A detailed account of the theory can be found in \cite{PS}. In this section, we recall some of the main results we shall need.

\subsection{Existence and uniqueness} Let $X_{1}\overset{d}{\hookrightarrow} X_{0}$ be a continuously and densely injected complex Banach couple.

\begin{definition}[Sectorial operators]
Let $\mathcal{P}(K,\theta)$, $\theta\in[0,\pi)$, $K\geq1$, be the class of all closed densely defined linear operators $A$ in $X_{0}$ such that 
\begin{gather}\label{sectrest}
S_{\theta}=\{\lambda\in\mathbb{C}\,|\, |\arg(\lambda)|\leq\theta\}\cup\{0\}\subset\rho{(-A)} \quad \mbox{and} \quad (1+|\lambda|)\|(A+\lambda)^{-1}\|_{\mathcal{L}(X_{0})}\leq K \quad \text{when} \quad \lambda\in S_{\theta}.
\end{gather}
The elements in $\mathcal{P}(\theta)=\cup_{K\geq1}\mathcal{P}(K,\theta)$ are called {\em (invertible) sectorial operators of angle $\theta$}.
\end{definition}

Note that if $A\in \mathcal{P}(\theta)$ then the Property \eqref{sectrest} can be extended from $S_{\theta}$ to the set $\{\lambda-\delta\in\mathbb{C}\, |\, \lambda\in S_{\widetilde{\theta}}\}$ for some $\delta>0$ and $\widetilde{\theta}\in(\theta,\pi)$, see e.g. \cite[(III.4.6.4)-(III.4.6.5)]{Am}. Therefore, whenever $A\in \mathcal{P}(\theta)$ we can always assume that $\theta>0$.

If $A\in\mathcal{P}(\theta)$, $\theta\in(0,\pi)$, then we can use the Dunford integral formula
\begin{equation}\label{dunfordformula}
\frac{1}{2\pi i}\int_{\Gamma_{R,\theta}}f(\lambda)(A+\lambda)^{-1} d\lambda
\end{equation}
with $f(\lambda)=(-\lambda)^{-z}$ to define complex powers $A^{-z}$, $\mathrm{Re}(z)>0$, of $A$. Here 
$$
\Gamma_{R,\theta}=\{ re^{-i\theta}\in\mathbb{C}\,|\,\infty>r\geq R\} \cup\{ Re^{i\phi}\in\mathbb{C}\,|\,2\pi-\theta\ge\phi\ge\theta\} \cup\{ re^{+i\theta}\in\mathbb{C}\,|\,R\le r<\infty\}
$$
and $R>0$ is such that the ball with center at the origin and radius $R>0$ is contained in $\rho(-A)$. The operators $A^{-z}\in\mathcal{L}(X_{0})$ are injections and allow the definition of $A^{z}=(A^{-z})^{-1}$, which are in general unbounded operators, see \cite[(III.4.6.12)]{Am}. By using Cauchy's integral formula we can deform the formula \eqref{dunfordformula} and define the imaginary powers $A^{it}$, $t\in\mathbb{R}\backslash\{0\}$, as the closure of the operator
$$
 \mathcal{D}(A)\ni x\mapsto A^{it}x=\frac{\sin(i\pi t)}{i\pi t}\int_{0}^{+\infty}s^{it}(A+s)^{-2}Ax ds,
$$
see e.g. \cite[(III.4.6.12)]{Am}. For the properties of the complex powers of a sectorial operator we refer to \cite[Theorem III.4.6.5]{Am}. Concerning the imaginary powers, the following property can be satisfied.

\begin{definition}[Bounded imaginary powers] Let $A\in\mathcal{P}(0)$ in $X_{0}$ and assume that there exists some $\delta,M>0$ such that $A^{it}\in \mathcal{L}(X_{0})$, $t\in(-\delta,\delta)$, and $\|A^{it}\|_{\mathcal{L}(X_{0})}\leq M$ when $t\in(-\delta,\delta)$. Then, $A^{it}\in \mathcal{L}(X_{0})$ for each $t\in\mathbb{R}$ and there exists some $\phi,\widetilde{M}>0$ such that $\|A^{it}\|_{\mathcal{L}(X_{0})}\leq \widetilde{M}e^{\phi|t|}$, $t\in\mathbb{R}$; in this case we say that {\em $A$ has bounded imaginary powers} and denote $A\in\mathcal{BIP}(\phi)$.
\end{definition}

Recall that if $A\in \mathcal{P}(\theta)$ with $\theta>\frac{\pi}{2}$ then $-A$ generates an analytic semigroup on $X_{0}$. The semigroup $\{e^{-tA}\}_{t\geq0}$ can be defined by \eqref{dunfordformula} with the choice $f(\lambda)=e^{t\lambda}$; it extends analytically to a sector of angle $\theta-\frac{\pi}{2}$.

Let $q\in(1,\infty)$, $T>0$, and consider the following abstract first order Cauchy problem
\begin{eqnarray}\label{app1}
u'(t)+Au(t)&=&w(t), \quad t\in(0,T),\\\label{app2}
u(0)&=&0,
\end{eqnarray}
where $-A:X_{1}\rightarrow X_{0}$ is the generator of an analytic semigroup on $X_{0}$ and $w\in L^q(0,T;X_{0})$. 

\begin{definition}
The operator $A$ has {\em maximal $L^q$-regularity} if for any $w$ there exists a unique 
$$
u\in W^{1,q}(0,T;X_{0})\cap L^{q}(0,T;X_{1})
$$ 
solving \eqref{app1}-\eqref{app2}. In this case, $u$ also depends continuously on $w$. Furthermore, the above property is independent of $q$ and $T$.
\end{definition}
 
If we restrict to the class of UMD (unconditionality of martingale differences property, see \cite[Section III.4.4]{Am}) Banach spaces, then the following result holds.

\begin{theorem}[{\rm Dore and Venni, \cite[Theorem 3.2]{DV}}]\label{dorevenni}
If $X_{0}$ is UMD and $A\in \mathcal{BIP}(\phi)$ with $\phi<\frac{\pi}{2}$ in $X_{0}$, then $A$ has maximal $L^{q}$-regularity. 
\end{theorem}

Note that in general we have $\mathcal{BIP}(\pi-\theta)\subset \mathcal{P}(\theta)$, $\theta\in(0,\pi)$. Hence, in the conditions of Dore and Venni Theorem, $-A$ is the generator of an analytic semigroup.

Let $q\in(1,\infty)$, $U$ be an open subset of $(X_{1},X_{0})_{\frac{1}{q},q}$, $-A\in \mathcal{L}(X_{1},X_{0})$ be the generator of an analytic semigroup and $F: U\times [0,T_{0}]\rightarrow X_{0}$ for some $T_{0}>0$. Consider the equation
\begin{eqnarray}\label{aqpp1}
u'(t)+Au(t)&=&F(u(t),t)+G(t),\quad t\in(0,T),\\\label{aqpp2}
u(0)&=&u_{0},
\end{eqnarray}
where $T\in(0,T_{0})$, $u_{0}\in U$ and $G\in L^{q}(0,T_{0};X_{0})$. Short time existence and uniqueness of solutions of the above equation is obtained by the following special case of a maximal $L^q$-regularity result by Cl\'ement and Li, namely of \cite[Theorem 2.1]{CL}. In the sequel, if $V\subseteq (X_{1},X_{0})_{\frac{1}{q},q}$ then by $F\in C^{1-}(V\times[0,T_{0}];X_{0})$ we mean
$$
\|F(v_{1},t_{1})-F(v_{2},t_{2})\| _{X_{0}}\le L(\|v_{1}-v_{2}\|_{(X_{1,}X_{0})_{\frac{1}{q},q}}+|t_{1}-t_{2}|),\quad v_{1}, v_{2}\in V, \, t_{1},t_{2}\in[0,T_{0}],
$$
for certain $L>0$ depending on $F$, $V$ and $T_{0}$.

\begin{theorem}[{\rm Cl\'ement and Li}]\label{thcl}
Assume that $A$ has maximal $L^{q}$-regularity and $F\in C^{1-}(U\times[0,T_{0}];X_{0})$. Then there exists a $T\in(0,T_{0}]$ and a unique $u\in W^{1,q}(0,T;X_{0})\cap L^{q}(0,T;X_{1})$ solving \eqref{aqpp1}-\eqref{aqpp2}.
\end{theorem}

\begin{remark}\label{embedtoC}
For each $q\in(1,\infty)$ and $T>0$ the following embedding holds
$$
W^{1,q}(0,T;X_{0})\cap L^{q}(0,T;X_{1})\hookrightarrow C([0,T];(X_{1},X_{0})_{\frac{1}{q},q}),
$$
see e.g. \cite[Theorem III.4.10.2]{Am}.
\end{remark}

In order to investigate the existence of global solutions, we define the maximal interval of existence $[0,T_{\max})$, where $T_{\max}>0$ is the supremum
of all $T\in(0,T_{0}]$ for which there is a $u\in W^{1,q}(0,T;X_{0})\cap L^{q}(0,T;X_{1})$ solving \eqref{aqpp1}-\eqref{aqpp2}. By the uniqueness provided by the Clement-Li theorem, there exists a unique function $u:[0,T_{\max})\to X_{0}$, called maximal solution, such that for all $T<T_{\max}$ we have that $u|_{[0,T)} \in W^{1,q}(0,T;X_{0})\cap L^{q}(0,T;X_{1})$ solves \eqref{aqpp1}-\eqref{aqpp2}.

\begin{corollary}[Global existence]\label{maximalinterval}
Assume that $A$ has maximal $L^{q}$-regularity, $G\in L^{q}(0,T_{0};X_{0})$ and $F:U\times[0,T_{0}]\to X_{0}$ is a function such that, for any $v\in U$ there exists a neighborhood $V\subseteq U$ containing $v$ such that $F|_{V\times[0,T_{0}]}\in C^{1-}(V\times[0,T_{0}];X_{0})$. Moreover, let $[0,T_{\max})$, $T_{\max}\le T_{0}$, be the maximal interval of existence and let $u:[0,T_{\max})\to X_{0}$ be the maximal solution of \eqref{aqpp1}-\eqref{aqpp2}. If $T_{\max}<T_{0}$, then one of the following two situations occurs:\\
{\em (i)} The limit $\lim_{t\to T_{\max}^{-}}u(t)$ exists in $(X_{1},X_{0})_{\frac{1}{q},q}$ but does not belong to $U$.\\
{\em (ii)} $\|F(u(\cdot),\cdot)\|_{L^{q}(0,T_{\max};X_{0})}=\infty$.
\end{corollary}

By the Corollary \ref{maximalinterval}, if $U$ is the whole space $(X_{1},X_{0})_{\frac{1}{q},q}$, then $T_{\max}<T_{0}$ can happen only if situation (ii) occurs.

Note also that for any $T<T_{\max}$ we have $u\in C([0,T];(X_{1},X_{0})_{\frac{1}{q},q})$. Therefore $[0,T]\ni t\mapsto F(u(t),t)\in X_{0}$
belongs to $C([0,T];X_{0})\subset L^{q}(0,T;X_{0})$ and condition (ii) is equivalent to
$$
\lim_{T\to T_{\max}^{-}}\|F(u(\cdot),\cdot)\| _{L^{q}(0,T;X_{0})}=\infty.
$$

\begin{proof}
Suppose that $T_{\max}<T_{0}$, let $T\in(0,T_{\max})$ and consider the problem 
\begin{eqnarray}\label{aqppA}
w'(t)+Aw(t)&=&F(u(t),t)+G(t),\quad t\in(0,T),\\\label{aqppB}
w(0)&=&u_{0}.
\end{eqnarray}
By the maximal $L^{q}$-regularity property of $A$, \eqref{aqppA}-\eqref{aqppB}
has a unique solution 
$$
w\in W^{1,q}(0,T;X_{0})\cap L^{q}(0,T;X_{1}).
$$
Clearly, $u$ is also a solution of \eqref{aqppA}-\eqref{aqppB}
in the above space, so that, by uniqueness, $u=w$.

By the maximal $L^{q}$-regularity inequality, see e.g. \cite[(2.2)]{CL},
we have that 
$$
\|u\|_{W^{1,q}(0,T;X_{0})\cap L^{q}(0,T;X_{1})}\leq C(\|F(u(\cdot),\cdot)+G(\cdot)\|_{L^{q}(0,T;X_{0})}+\|u_{0}\|_{(X_{1},X_{0})_{\frac{1}{q},q}}),
$$
for some constant $C>0$ depending on $A$, $q$ and $T_{max}$. Suppose that $F(u(\cdot),\cdot)\in L^{q}(0,T_{\max};X_{0})$. Then, by letting $T \rightarrow T_{\max}$,
we get that 
$$
\|u\|_{W^{1,q}(0,T_{\max};X_{0})\cap L^{q}(0,T_{\max};X_{1})}<\infty.
$$ 
In particular, by Remark \ref{embedtoC}, $u\in C([0,T_{\max}];(X_{1},X_{0})_{\frac{1}{q},q})$,
so that $u$ can be extended continuously up to $T_{\max}$. Let $u(T_{\max})=\lim_{t\to T_{\max}^{-}}u(t)$ in $(X_{1},X_{0})_{\frac{1}{q},q}$. If $u(T_{\max})\notin U$, then we are under the condition (i) of the Corollary. Let us show that we obtain a contradiction in the case $u(T_{\max})\in U$.

Consider the problem 
\begin{eqnarray}
v'(t)+Av(t) & = & F(v(t),t)+G(t),\quad t\ge T_{\max},\label{aqppC}\\
v(T_{\max}) & = & u(T_{\max}).\label{aqppD}
\end{eqnarray}
 Since $u(T_{\max})\in U$, by Theorem
\ref{thcl} there exists $T_{1}\in (T_{\max},T_{0})$ and a
unique 
$$
v\in W^{1,q}(T_{\max},T_{1};X_{0})\cap L^{q}(T_{\max},T_{1};X_{1})
$$
solving \eqref{aqppC}-\eqref{aqppD}.
Then, the function 
$$
f=\bigg\{\begin{array}{ccc}
u & \text{in} & [0,T_{\max}]\\
v & \text{in} & (T_{\max},T_{1}]
\end{array}
$$
 is a solution
of \eqref{aqpp1}-\eqref{aqpp2} in $W^{1,q}(0,T_{1};X_{0})\cap L^{q}(0,T_{1};X_{1})$,
contradicting the maximality of $[0,T_{\max})$.
\end{proof}

\subsection{Smoothing} 

Concerning smoothness in time for solutions of \eqref{aqpp1}-\eqref{aqpp2}, we recall the following theorem, which is a particular case of a result by Pr\"uss and Simonett \cite[Theorem 5.2.1]{PS}.

\begin{thm}[{\rm Pr\"uss and Simonett}]\label{regularity} Assume that $A:X_{1} \to X_{0}$ has maximal $L^{q}$-regularity, $F\in C^{\infty}((X_{1},X_{0})_{\frac{1}{q},q};X_{0})$, $G\equiv0$ and let $u\in W^{1,q}(0,T;X_{0})\cap L^{q}(0,T;X_{1})$ be the unique solution of \eqref{aqpp1}-\eqref{aqpp2} for some $T>0$. Then $u\in C^{\infty}((0,T);X_{1})$.
\end{thm}

Next we combine the above theorem together with \cite[Theorem 3.1]{RS4} in order to obtain a space-time smoothing result. Consider two complex Banach scales $\{ Y_{0}^{k}\} _{k\in\mathbb{N}}$ and $\{ Y_{1}^{k}\} _{k\in\mathbb{N}}$
 such that for each $k\in\mathbb{N}$ we have $Y_{0}^{k+1}\overset{d}{\hookrightarrow}Y_{0}^{k}$, $Y_{1}^{k+1}\overset{d}{\hookrightarrow}Y_{1}^{k}$, $Y_{1}^{k}\overset{d}{\hookrightarrow}Y_{0}^{k}$, 
 and $Y_{1}^{k}\overset{d}{\hookrightarrow}(Y_{1}^{k+1},Y_{0}^{k+1})_{\frac{1}{q},q}$, for some fixed $q\in(1,\infty)$. Moreover, consider the equation
\begin{eqnarray}\label{ACP}
u'(t)+Au(t)&=&F(u(t)),\quad t>0,\\\label{ACP2}
u(0)&=&u_{0},
\end{eqnarray}
where $-A\in \mathcal{L}(Y_{1}^{0},Y_{0}^{0})$ is the generator of an analytic semigroup, $F:(Y_{1}^{0},Y_{0}^{0})_{\frac{1}{q},q}\to Y_{0}^{0}$ is a suitable map and $u_{0}\in(Y_{1}^{0},Y_{0}^{0})_{\frac{1}{q},q}$.

\begin{thm}[Smoothing]\label{abstractshortsmooth}
 Suppose that for each $k\in\mathbb{N}${\em:}\\
{\em (i)} $A(Y_{1}^{k})\subset Y_{0}^{k}$ and $A:Y_{1}^{k}\to Y_{0}^{k}$ has maximal $L^{q}$-regularity.\\
{\em (ii)} $F|_{(Y_{1}^{k},Y_{0}^{k})_{\frac{1}{q},q}}\in C^{\infty}((Y_{1}^{k},Y_{0}^{k})_{\frac{1}{q},q};Y_{0}^{k})$ and $F|_{Y_{1}^{k}}\in C(Y_{1}^{k};Y_{0}^{k+1})$.\\
Then there exists $T>0$ and a unique function $u\in W^{1,q}(0,T;Y_{0}^{0})\cap L^{q}(0;T,Y_{1}^{0})$ solving \eqref{ACP}-\eqref{ACP2}. Moreover if $u_{\max}:[0,T_{\max})\to Y_{0}^{0}$ is the maximal solution, then $u_{\max}\in C^{\infty}((0,T_{\max});Y_{1}^{k})$ for all $k\in\mathbb{N}$.
\end{thm}
\begin{proof}
The assumptions for $k=0$ together with Theorem \ref{thcl} imply the existence a unique maximal solution $u_{\max}:[0,T_{\max})\to Y_{0}^{0}$ of \eqref{ACP}-\eqref{ACP2}; it satisfies 
$$
u:=u_{\max}|_{[0,T)}\in W^{1,q}(0,T;Y_{0}^{0})\cap L^{q}(0,T;Y_{1}^{0}),
$$
for each $T<T_{\max}$. Moreover, Theorem \ref{regularity} implies that $u\in C^{\infty}((0,T_{\max}),Y_{1}^{0})$.

In order to obtain higher regularity, we argue by induction. Suppose that $u\in C^{\infty}((0,T_{\max}),Y_{1}^{k})$ for some $k\in\mathbb{N}$. For each $t_{0}\in(0,T_{\max})$ we have that $u(t_{0})\in Y_{1}^{k}\overset{d}{\hookrightarrow}(Y_{1}^{k+1},Y_{0}^{k+1})_{\frac{1}{q},q}$. By Theorem \ref{thcl}, there exists $\widetilde{T}>t_{0}$ and a unique $\widetilde{u}\in W^{1,q}(t_{0},\widetilde{T};Y_{0}^{k+1})\cap L^{q}(t_{0},\widetilde{T};Y_{1}^{k+1})$ solving the equation 
\begin{eqnarray}\label{eqwrt1}
\widetilde{u}'(t)+A\widetilde{u}(t) & = & F(\widetilde{u}(t)),\quad t>t_{0},\\\label{eqwrt2}
\widetilde{u}(t_{0}) & = & u(t_{0}).
\end{eqnarray}

Let $(t_{0},\widetilde{T}_{\max})$ be the maximal interval of existence of the solution of the above equation and let $\widetilde{u}_{\max}:[t_{0},\widetilde{T}_{\max})\to Y_{0}^{k+1}$ be the maximal solution.

Suppose $\widetilde{T}_{\max}<T_{\max}$. By uniqueness of Theorem \ref{thcl} we have that $\widetilde{u}_{\max}=u|_{[t_{0},\widetilde{T}_{\max})}$ and therefore $\widetilde{u}_{\max}\in C([t_{0},\widetilde{T}_{\max}],Y_{1}^{k})$. Hence, by the assumption (ii), 
$$
F(\widetilde{u}(\cdot))\in C([t_{0},\widetilde{T}_{\max}];Y_{0}^{k+1})\hookrightarrow L^{q}(t_{0},\widetilde{T}_{\max};Y_{0}^{k+1}).
$$
Due to Corollary \ref{maximalinterval}, we conclude that $\widetilde{u}_{\max}$ can be extended to a larger interval, which is a contradiction. Therefore, $\widetilde{u}_{\max}$ is well defined for all $t\in[t_{0},T_{\max})$. Moreover, by Theorem \ref{regularity} applied to \eqref{eqwrt1}-\eqref{eqwrt2}, $\widetilde{u}_{\max}\in C^{\infty}((t_{0},T_{\max});Y_{1}^{k+1})$. Hence, $u|_{(t_{0},T_{\max})}\in C^{\infty}((t_{0},T_{\max});Y_{1}^{k+1})$. As $t_{0}>0$ is arbitrary, we conclude that $u\in C^{\infty}((0,T_{\max}),Y_{1}^{k+1})$.
\end{proof}

\section{Manifolds with conical singularities}

We study the Laplacian and bi-Laplacian on a conic manifold by regarding them as elements in the class of {\em cone differential operators}, i.e. the naturally appearing operators in such spaces. In this section we recall the main properties of these operators and of the associated function spaces. For more details we refer to \cite{GKM}, \cite{Le}, \cite{RS2}, \cite{RS3}, \cite{RS1}, \cite{SS}, \cite{Schu} and \cite{Sei}.

\subsection{Cone differential operators and Mellin-Sobolev spaces.} A cone differential operator of order $\mu\in\mathbb{N}$ is an $\mu$-th order differential operator $A$ with smooth coefficients in the interior $\mathbb{B}^{\circ}$ of $\mathbb{B}$ such that, in local coordinates $(x,y)\in(0,1)\times\partial\mathcal{B}$ on the collar part, it admits the following form 
\begin{gather}\label{Acone}
A=x^{-\mu}\sum_{k=0}^{\mu}a_{k}(x)(-x\partial_{x})^{k}, \quad \mbox{where} \quad a_{k}\in C^{\infty}([0,1);\mathrm{Diff}^{\mu-k}(\partial\mathbb{B})).
\end{gather}
In addition to the usual pseudodifferential symbol, we define the {\em rescaled symbol} of $A$ by
$$
T^{*}(\partial\mathbb{B})\times\mathbb{R} \ni (y,\xi,\tau)\mapsto\sum_{k=0}^{\mu}\sigma_{\psi}^{\mu-k}(a_{k})(0,y,\xi)(-i\tau)^{k},
$$
where $\sigma_{\psi}^{\mu-k}(a_{k})\in C^{\infty}([0,1)\times T^{*}(\partial\mathbb{B}))$ is the principal symbol of $a_{k}$ (see e.g. \cite[(2.3)]{CSS1} for more information). If it is also pointwise invertible when $\xi\neq0$ or $\tau\neq0$, then $A$ is called {\em $\mathbb{B}$-elliptic}; this is the case for the Laplacian $\Delta$ and bi-Laplacian $\Delta^2$.

Cone differential operators act naturally on scales of {\em Mellin-Sobolev} spaces. Let $\omega\in C^{\infty}(\mathbb{B})$ be a fixed cut-off function near the boundary, i.e. a smooth non-negative function on $\mathcal{B}$ with $\omega=1$ near $\{0\}\times\partial \mathcal{B}$ and $\omega=0$ on $\mathcal{B}\backslash([0,1)\times \partial \mathcal{B})$. Moreover, assume that in the local coordinates $(x,y)\in [0,1)\times \partial\mathcal{B}$, $\omega$ depends only on $x$. Denote by $C_{c}^{\infty}$ the space of smooth compactly supported functions and by $H_{p}^{s}$, $p\in(1,\infty)$, $s\in\mathbb{R}$, the usual Sobolev spaces. In addition, denote by $\mathbb{R}_{+}$ the set $(0,\infty)$.

\begin{definition}[Mellin-Sobolev spaces] \label{mellinsobolevspaces}
For any $\gamma\in\mathbb{R}$ consider the map 
$$
M_{\gamma}: C_{c}^{\infty}(\mathbb{R}_{+}\times\mathbb{R}^{n})\rightarrow C_{c}^{\infty}(\mathbb{R}^{n+1}) \quad \mbox{defined by} \quad u(x,y)\mapsto e^{(\gamma-\frac{n+1}{2})x}u(e^{-x},y). 
$$
Furthermore, take a covering $\kappa_{j}:U_{j}\subseteq\partial\mathcal{B} \rightarrow\mathbb{R}^{n}$, $j\in\{1,...,N\}$, $N\in\mathbb{N}\backslash\{0\}$, of $\partial\mathcal{B}$ by coordinate charts and let $\{\phi_{j}\}_{j\in\{1,...,N\}}$ be a subordinated partition of unity. For any $s\in\mathbb{R}$ and $p\in(1,\infty)$ let $\mathcal{H}^{s,\gamma}_p(\mathbb{B})$ be the space of all distributions $u$ on $\mathbb{B}^{\circ}$ such that 
\begin{equation}
\label{mellinsobolev}
\|u\|_{\mathcal{H}^{s,\gamma}_p(\mathbb{B})}=\sum_{j=1}^{N}\|M_{\gamma}(1\otimes \kappa_{j})_{\ast}(\omega\phi_{j} u)\|_{H^{s}_p(\mathbb{R}^{n+1})}+\|(1-\omega)u\|_{H^{s}_p(\mathbb{B})}
\end{equation}
is defined and finite, where $\ast$ refers to the push-forward of distributions. The space $\mathcal{H}^{s,\gamma}_p(\mathbb{B})$, called {\em (weighted) Mellin-Sobolev space}, is independent of the choice of the cut-off function $\omega$, the covering $\{\kappa_{j}\}_{j\in\{1,...,N\}}$ and the partition $\{\phi_{j}\}_{j\in\{1,...,N\}}$; if $A$ is as in \eqref{Acone}, then it induces a bounded map
$$
A: \mathcal{H}^{s+\mu,\gamma+\mu}_p(\mathbb{B}) \rightarrow \mathcal{H}^{s,\gamma}_p(\mathbb{B}).
$$
Finally, if $s\in \mathbb{N}$, then equivalently, $\mathcal{H}^{s,\gamma}_p(\mathbb{B})$ is the space of all functions $u$ in $H^s_{p,loc}(\mathbb{B}^\circ)$ such that near the boundary
\begin{equation}\label{mellinsobolevinteger}
x^{\frac{n+1}2-\gamma}(x\partial_x)^{k}\partial_y^{\alpha}(\omega(x) u(x,y)) \in L^{p}([0,1)\times \partial \mathcal{B}, \sqrt{\mathrm{det}[h(x)]}\frac{dx}xdy),\quad k+|\alpha|\le s.
\end{equation}
\end{definition}

Next we collect some elementary properties of Mellin-Sobolev spaces.

\begin{lemma}\label{propms}
{\em (i)} $\mathcal{H}_{p}^{s,\gamma}(\mathbb{B})$, $p\in(1,\infty)$, $s,\gamma\in\mathbb{R}$, is a UMD Banach space; in particular, it is a Hilbert space if $p=2$. Moreover, $L^{p}(\mathbb{B})=\mathcal{H}_{p}^{0,(n+1)(\frac{1}{2}-\frac{1}{p})}(\mathbb{B})$, where $L^{p}(\mathbb{B})$ is defined by using the Riemannian measure $d\mu_{g}$ induced by the conic metric $g$.\\
{\em (ii)} For every $p\in(1,\infty)$ and $s,\gamma\in\mathbb{R}$, the bilinear map 
$$
\langle \cdot,\cdot\rangle :C^{\infty}_{c}(\mathbb{B}^{\circ})\times C^{\infty}_{c}(\mathbb{B}^{\circ})\to\mathbb{C} \quad \text{given by}\quad \langle w,v\rangle =\int_{\mathbb{B}}wvd\mu_{g}
$$ 
has a unique extension to a continuous bilinear map 
$$
\langle \cdot,\cdot\rangle _{\mathcal{H}_{p}^{s,\gamma}\times\mathcal{H}_{q}^{-s,-\gamma}}:\mathcal{H}_{p}^{s,\gamma}(\mathbb{B})\times\mathcal{H}_{q}^{-s,-\gamma}(\mathbb{B})\to\mathbb{C},\quad \frac{1}{p}+\frac{1}{q}=1.
$$
This map, called duality map, allows the identification of the dual of $\mathcal{H}_{p}^{s,\gamma}(\mathbb{B})$ with $\mathcal{H}_{q}^{-s,-\gamma}(\mathbb{B})$ in a standard
way.\\
{\em (iii)} If $u\in\mathcal{H}_{p}^{s,\gamma}(\mathbb{B})$ with $p\in(1,\infty)$, $s>\frac{n+1}{p}$ and $\gamma\in\mathbb{R}$, then $u\in C(\mathbb{B}^{\circ})$ and in local coordinates on the collar part we have
$$
|u(x,y)|\le Cx^{\gamma-\frac{n+1}{2}}\Vert u\Vert _{\mathcal{H}_{p}^{s,\gamma}(\mathbb{B})}, \quad (x,y)\in(0,1)\times\partial\mathcal{B},
$$
with some constant $C>0$ that depends only on $\mathbb{B}$ and $p$. Moreover, we denote 
$$
\mathcal{H}_{p}^{\infty,\gamma}(\mathbb{B})=\bigcap_{\nu\ge0}\mathcal{H}_{p}^{\nu,\gamma}(\mathbb{B})\subset C^{\infty}(\mathbb{B}^{\circ}).
$$
\\
{\em (iv)} If $q\ge p$, $s\ge t+(n+1)(\frac{1}{p}-\frac{1}{q})$ and $\gamma_{1}\ge\gamma_{2}$, then $\mathcal{H}_{p}^{s,\gamma_{1}}(\mathbb{B})\hookrightarrow\mathcal{H}_{q}^{t,\gamma_{2}}(\mathbb{B})$.
\\
{\em (v)} If $q\le p$, $s\ge t\ge 0$ and $\gamma_{1}>\gamma_{2}$, then $\mathcal{H}_{p}^{s,\gamma_{1}}(\mathbb{B})\hookrightarrow\mathcal{H}_{q}^{t,\gamma_{2}}(\mathbb{B})$.
\\
{\em (vi)} For any $p,q\in(1,\infty)$ and $\gamma\in\mathbb{R}$, we have that 
\begin{equation}\label{intersectspace}
\bigcap_{\varepsilon>0}\mathcal{H}_{p}^{\infty,\gamma-\varepsilon}(\mathbb{B})=\bigcap_{\varepsilon>0}\mathcal{H}_{q}^{\infty,\gamma-\varepsilon}(\mathbb{B}).
\end{equation}
In addition,
$$
\mathcal{C}^{\infty,\gamma-}(\mathbb{B}):=\bigcap_{\varepsilon>0}\mathcal{H}_{p}^{\infty,\gamma-\varepsilon}(\mathbb{B}),
$$
coincides with the space of all functions $u\in C^{\infty}(\mathbb{B})$ such that
\begin{equation}\label{flatspace}
\sup_{(x,y)\in(0,1)\times\partial\mathbb{B}}|x^{\frac{n+1}{2}-\gamma+\varepsilon}(x\partial_{x})^{k}\partial_{y}^{\alpha}u(x,y)|<\infty, 
\end{equation}
for all $k$, $\alpha$ and $\varepsilon>0$.
\end{lemma}

\begin{proof}
The properties can be proved by localization using \eqref{mellinsobolev}. Parts (i), (ii) and (iv) are stated in \cite{RS4} and \cite{Sei}; part (iii) is \cite[Corollary 2.5]{RS1}.

Part (v) is probably well known, although we could not find it explicitly in the literature. We give here a proof for the convenience of the reader. Let us prove it by analyzing the terms defining the norm \eqref{mellinsobolev}. We first focus on $\|(1-\omega)u\|_{H_{p}^{s}(\mathbb{B})}$.

Using H\"older inequality and the fact that $(1-\omega)u$ has compact support, we conclude that there is a constant $C_{1}>0$ such
that 
\begin{equation}
\|(1-\omega)u\|_{H_{q}^{t}(\mathbb{B})}\le C_{1}\|(1-\omega)u\|_{H_{p}^{s}(\mathbb{B})},\quad u\in\mathcal{H}_{p}^{s,\gamma_{1}}(\mathbb{B}),\label{eq:First term}
\end{equation}
whenever $s=t\in\mathbb{N}$. The inequality \eqref{eq:First term}
for all $s=t\ge0$ follows by complex interpolation. Finally, for
$s\ge t\ge0$, we use the Sobolev inclusion $H_{q}^{s}\hookrightarrow H_{q}^{t}$.

In order to study the term $\|M_{\gamma_{2}}(1\otimes\kappa_{i})_{\ast}(\omega\phi_{i}u)\|_{H_{q}^{t}(\mathbb{R}^{n+1})}$,
we write it in local coordinates 
$$
e^{(\gamma_{2}-\frac{n+1}{2})x}u(e^{-x},y)\phi_{i}(y)\omega(e^{-x})=e^{(\gamma_{2}-\gamma_{1})x}e^{(\gamma_{1}-\frac{n+1}{2})x}u(e^{-x},y)\phi_{i}(y)\omega(e^{-x}).
$$
Due to the cut-off function $\omega$, we can consider only $x>0$. In this
case, the function $(0,\infty)\ni x\mapsto e^{(\gamma_{2}-\gamma_{1})x}$
and all its derivatives are bounded. Moreover this function and its
derivatives belong to $L^{r}(0,\infty)$ for
all $1<r<\infty$. Under these considerations, from complex interpolation and H\"older inequality we conclude that, if $p\ge q$ and $s\ge t\ge 0$, then there is a $C_{2}>0$ such that 
\begin{equation}
\|M_{\gamma_{2}}(1\otimes\kappa_{i})_{\ast}(\omega\phi_{i}u)\|_{H_{q}^{t}(\mathbb{R}^{n+1})}\le C_{2}\|M_{\gamma_{1}}(1\otimes\kappa_{i})_{\ast}(\omega\phi_{i}u)\|_{H_{p}^{s}(\mathbb{R}^{n+1})},\quad u\in\mathcal{H}_{p}^{s,\gamma_{1}}(\mathbb{B}).\label{eq:second term}
\end{equation}
This concludes the proof. 

We note that part (iv) stated in \cite{Sei} can be proved in a similar way by using the fact that $H_{p}^{s}(\mathbb{R}^{n+1})\hookrightarrow H_{q}^{t}(\mathbb{R}^{n+1})$
whenever $s\ge t+(n+1)(\frac{1}{p}-\frac{1}{q})$ and $p\le q$, see \cite[Theorem 6.5.1]{BL}.

Part (vi) is a simple consequence of parts (iv) and (v). In fact, parts (iv) and (v) imply that, for any $s>0$ and $\varepsilon>0$, we have
$$
\mathcal{H}_{p}^{\infty,\gamma-\frac{\varepsilon}{2}}(\mathbb{B})\subseteq\mathcal{H}_{p}^{\max\{s,s+(n+1)(\frac{1}{p}-\frac{1}{q})\},\gamma-\frac{\varepsilon}{2}}(\mathbb{B})\subseteq\mathcal{H}_{q}^{s,\gamma-\varepsilon}(\mathbb{B}).
$$
After taking the intersection first in $s$ and then in $\varepsilon$, we obtain 
$$
\bigcap_{\varepsilon>0}\mathcal{H}_{p}^{\infty,\gamma-\varepsilon}(\mathbb{B})\subseteq\bigcap_{\varepsilon>0}\mathcal{H}_{q}^{\infty,\gamma-\varepsilon}(\mathbb{B}).
$$
As $p$ and $q$ are arbitrary numbers in $(1,\infty)$, we can change their roles in the above argument and obtain \eqref{intersectspace}. Finally, \eqref{flatspace} follows by (iii), \eqref{intersectspace} and the discussion in \cite[Section 2.3]{Sei}.
\end{proof}

Concerning the interpolation of Mellin-Sobolev spaces, the following properties hold, where, for $\xi\in(0,1)$, we denote by $[\cdot,\cdot]_{\xi}$ the complex interpolation functor.

\begin{lemma}\label{prop:interpolation}
Let $s_{0}, s_{1}, \gamma_{0}, \gamma_{1}\in \mathbb{R}$, $\theta\in(0,1)$, $q\in(1,\infty)$ and denote $s=(1-\theta)s_{0}+\theta s_{1}$, $\gamma=(1-\theta)\gamma_{0}+\theta\gamma_{1}$. Then, for any $\varepsilon_{1}$, $\varepsilon_{2}>0$ we have:\\
{\em (i)} 
$$
(\mathcal{H}_{p}^{s_{0},\gamma_{0}}(\mathbb{B}),\mathcal{H}_{p}^{s_{1},\gamma_{1}}(\mathbb{B}))_{\theta,q}\hookrightarrow\bigg\{ \begin{array}{lll}
\mathcal{H}_{p}^{s,\gamma-\varepsilon_{2}}(\mathbb{B}) & \text{if} & q\le2\\
\mathcal{H}_{p}^{s-\varepsilon_{1},\gamma-\varepsilon_{2}}(\mathbb{B}) & \text{if} &q>2.
\end{array}
$$\\
{\em (ii)} $\mathcal{H}_{p}^{s+\varepsilon_{1},\gamma+\varepsilon_{2}}(\mathbb{B})\hookrightarrow(\mathcal{H}_{p}^{s_{0},\gamma_{0}}(\mathbb{B}),\mathcal{H}_{p}^{s_{1},\gamma_{1}}(\mathbb{B}))_{\theta,q}$.\\
{\em (iii)} $[\mathcal{H}_{p}^{s_{0},\gamma_{0}}(\mathbb{B}),\mathcal{H}_{p}^{s_{1},\gamma_{1}}(\mathbb{B})]_{\theta}=\mathcal{H}_{p}^{s,\gamma}(\mathbb{B})$,
whenever $s_{1}\ge s_{0}$ and $\gamma_{1}\ge\gamma_{0}$.
\end{lemma}
\begin{proof}
Part (i) was proved by \cite[Lemma 5.4]{CSS1}. Part (ii) can be proved similarly and was stated in \cite[Lemma 3.6]{RS3}. 

Concerning part (ii), the case of $\gamma_{0}=\gamma_{1}$ was proved by \cite[Lemma 3.7]{RS3}. Assume that $\gamma_{1}>\gamma_{0}$. Let $X_{0}$ and $X_{1}$ be Banach spaces such that $X_{1}$ is densely and continuously immersed into $X_{0}$, i.e. $X_{1}\overset{d}{\hookrightarrow}X_{0}$. We denote by $\mathcal{A}(X_{0},X_{1})$ the set of all bounded continuous functions $f:Z=\{ \lambda\in\mathbb{C}\,|\,0\leq\mathrm{Re}(\lambda)\leq1\} \rightarrow X_{0}$
that are analytic in the interior $Z^{\circ}$ of $Z$, that satisfy $f(1+it)\in X_{1}$, $t\in\mathbb{R}$, and define a continuous and bounded map
$\mathbb{R}\ni t\mapsto f(1+it)\in X_{1}$.

Let $(x,y)\in[0,1)\times \partial\mathcal{B}$ be local coordinates on the collar part of $\mathbb{B}$. Fix a $C^{\infty}(\mathbb{B}^{\circ})$-function
$\mathsf{x}$ that is equal to $x$ on $[0,1/2)\times\partial\mathcal{B}$
and satisfies $\mathsf{x}\geq1/4$ on $\mathcal{B}\backslash([0,1/2)\times\partial\mathcal{B})$.
Clearly, for $\mu>0$, $\widetilde{p}\in(1,\infty)$ and $\widetilde{s}, \widetilde{\gamma}\in\mathbb{R}$, the multiplication function
$\mathcal{H}_{\widetilde{p}}^{\widetilde{s},\widetilde{\gamma}}(\mathbb{B})\ni u\mapsto\mathsf{x}^{\mu z}u\in\mathcal{H}_{\widetilde{p}}^{\widetilde{s},\widetilde{\gamma}}(\mathbb{B})$,
denoted by $\mathsf{x}^{\mu z}$, defines an analytic map $\{ w\in\mathbb{C}\, |\,\text{Re}(w)>0\} \ni z\mapsto\mathsf{x}^{\mu z}\in\mathcal{L}(\mathcal{H}_{\widetilde{p}}^{\widetilde{s},\widetilde{\gamma}}(\mathbb{B}))$. Indeed, for each $z,h\in \mathbb{C}$ with $\mathrm{Re}(z)>0$ and $|h|$ sufficiently small, by \cite[Corollary 3.3]{RS3} we have
$$
\|\big(\frac{\mathsf{x}^{\mu (z+h)}-\mathsf{x}^{\mu z}}{h}-\mu\ln(x)\mathsf{x}^{\mu z}\big)u\|_{\mathcal{H}_{\widetilde{p}}^{\widetilde{s},\widetilde{\gamma}}(\mathbb{B})}\leq C_{1}\|\mathsf{x}^{\mu z}\big(\frac{\mathsf{x}^{\mu h}-1}{h}-\mu\ln(x)\big)\|_{\mathcal{H}_{\widetilde{p}}^{k,(n+1)/2}(\mathbb{B})}\|u\|_{\mathcal{H}_{\widetilde{p}}^{\widetilde{s},\widetilde{\gamma}}(\mathbb{B})},
$$
for certain $C_{1}>0$, where $k\in\mathbb{N}$ is fixed and sufficiently large. By the dominated convergence theorem, the first term on the right hand side of the above inequality tends to zero as $h\rightarrow0$. This shows the claim. 
 
Moreover, if $u\in\mathcal{H}_{\widetilde{p}}^{\widetilde{s},\widetilde{\gamma}}(\mathbb{B})$ then the following estimate holds
\begin{equation}
\Vert \mathsf{x}^{\mu z}u\Vert _{\mathcal{H}_{\widetilde{p}}^{\widetilde{s},\widetilde{\gamma}+\mu\mathrm{Re}(z)}(\mathbb{B})}\le C_{2}(1+|z|)^{|\widetilde{s}|}\Vert u\Vert _{\mathcal{H}_{\widetilde{p}}^{\widetilde{s},\widetilde{\gamma}}(\mathbb{B})},\quad\text{Re}(z)\ge0,\label{eq:estxi}
\end{equation}
for certain $C_{2}>0$. Inequality \eqref{eq:estxi} is clear for $\widetilde{s}\in\mathbb{N}$ due to \eqref{mellinsobolev} and \eqref{mellinsobolevinteger}. As we have seen, the third part of the lemma holds when $\gamma_{0}=\gamma_{1}$. Therefore \eqref{eq:estxi} holds for every $\widetilde{s}\ge0$ by complex interpolation and for all $\widetilde{s}\in\mathbb{R}$ by duality due to Lemma \ref{propms} (ii).
 
Similarly, if $t,\ell\in\mathbb{R}$, then for any $j\in\{0,1\}$ the term
$$
\|(\mathsf{x}^{\mu (j+i(t+\ell))}-\mathsf{x}^{\mu(j+it)})u\|_{\mathcal{H}_{\widetilde{p}}^{\widetilde{s},\widetilde{\gamma}+j\mu}(\mathbb{B})}=\|(\mathsf{x}^{i\mu\ell}-1)\mathsf{x}^{\mu (j+it)}u\|_{\mathcal{H}_{\widetilde{p}}^{\widetilde{s},\widetilde{\gamma}+j\mu}(\mathbb{B})}
$$
 tends to zero as $\ell\rightarrow0$. Indeed, when $\widetilde{s}\in\mathbb{N}$, the above follows by the dominated convergence theorem. Again by interpolation and duality, we obtain the result for any $s\in\mathbb{R}$. This shows that for any $k>|\widetilde{s}|$ and $j\in\{0,1\}$, the map
$$
\mathbb{R}\ni t \mapsto\frac{\mathsf{x}^{\mu (j+it)}}{(1+j+it)^{k}}\in\mathcal{L}_{s}(\mathcal{H}_{\widetilde{p}}^{\widetilde{s},\widetilde{\gamma}}(\mathbb{B}),\mathcal{H}_{\widetilde{p}}^{\widetilde{s},\widetilde{\gamma}+j\mu}(\mathbb{B}))
$$
is continuous and bounded.

As a consequence of the previous remarks, for $k>\max\{ |s_{0}|,|s_{1}|\}$ we have the following inclusion
\begin{eqnarray}\nonumber
 \lefteqn{[\mathcal{H}_{p}^{s_{0},\gamma_{0}}(\mathbb{B}),\mathcal{H}_{p}^{s_{1},\gamma_{1}}(\mathbb{B})]_{\theta}}\\\nonumber
 & = & \{ f(\theta)\,|\,f\in\mathcal{A}(\mathcal{H}_{p}^{s_{0},\gamma_{0}}(\mathbb{B}),\mathcal{H}_{p}^{s_{1},\gamma_{1}}(\mathbb{B}))\} \\\nonumber
 & \hookleftarrow & \{ f(\theta)\,|\,f=\frac{\mathsf{x}^{(\gamma_{1}-\gamma_{0})z}}{(1+z)^{k}}w,\,w\in\mathcal{A}(\mathcal{H}_{p}^{s_{0},\gamma_{0}}(\mathbb{B}),\mathcal{H}_{p}^{s_{1},\gamma_{0}}(\mathbb{B}))\} \\\nonumber
 & = & \frac{\mathsf{x}^{(\gamma_{1}-\gamma_{0})\theta}}{(1+\theta)^{k}}\{ w(\theta)\,|\,w\in\mathcal{A}(\mathcal{H}_{p}^{s_{0},\gamma_{0}}(\mathbb{B}),\mathcal{H}_{p}^{s_{1},\gamma_{0}}(\mathbb{B}))\} \\\nonumber
 & = & \frac{\mathsf{x}^{(\gamma_{1}-\gamma_{0})\theta}}{(1+\theta)^{k}}[\mathcal{H}_{p}^{s_{0},\gamma_{0}}(\mathbb{B}),\mathcal{H}_{p}^{s_{1},\gamma_{0}}(\mathbb{B})]_{\theta} \\\label{leftemb}
 & = & \frac{\mathsf{x}^{(\gamma_{1}-\gamma_{0})\theta}}{(1+\theta)^{k}}\mathcal{H}_{p}^{s,\gamma_{0}}(\mathbb{B})=\mathcal{H}_{p}^{s,\gamma}(\mathbb{B}),
\end{eqnarray}
 where we have used \cite[Lemma 3.7]{RS3}.

By Lemma \ref{propms} (ii), the Mellin-Sobolev
spaces are reflexive. Since $C_{c}^{\infty}(\mathbb{B}^{\circ})$
is dense in both $\mathcal{H}_{p}^{s_{0},\gamma_{0}}(\mathbb{B})$
and $\mathcal{H}_{p}^{s_{1},\gamma_{1}}(\mathbb{B})$,
by \eqref{leftemb} and \cite[Theorems
1.9.3 (b) and 1.11.3]{Trib}
we obtain, for any $q\in(1,\infty)$, $p$ such that $\frac{1}{q}+\frac{1}{p}=1$, $\gamma_{1}>\gamma_{0}$
and $s_{1}\ge s_{0}$, that
\begin{eqnarray*}
\lefteqn{[\mathcal{H}_{q}^{-s_{1},-\gamma_{1}}(\mathbb{B}),\mathcal{H}_{q}^{-s_{0},-\gamma_{0}}(\mathbb{B})]_{1-\theta}	=	[\mathcal{H}_{q}^{-s_{0},-\gamma_{0}}(\mathbb{B}),\mathcal{H}_{q}^{-s_{1},-\gamma_{1}}(\mathbb{B})]_{\theta}}\\
&	\cong &	[\mathcal{H}_{p}^{s_{0},\gamma_{0}}(\mathbb{B})',\mathcal{H}_{p}^{s_{1},\gamma_{1}}(\mathbb{B})']_{\theta}\cong[\mathcal{H}_{p}^{s_{0},\gamma_{0}}(\mathbb{B}),\mathcal{H}_{p}^{s_{1},\gamma_{1}}(\mathbb{B})]_{\theta}'\\
& \hookrightarrow &	\mathcal{H}_{p}^{s,\gamma}(\mathbb{B})'\cong\mathcal{H}_{q}^{-s,-\gamma}(\mathbb{B}),
\end{eqnarray*}
where by $\cong$ we mean norm equivalence. In particular, performing the changes $q\mapsto\widetilde{p}$, $-s_{1}\mapsto\widetilde{s}_{0}$,
$-s_{0}\mapsto\widetilde{s}_{1}$, $-\gamma_{1}\mapsto\widetilde{\gamma}_{0}$,
$-\gamma_{0}\mapsto\widetilde{\gamma}_{1}$ and $\theta\mapsto1-\widetilde{\theta}$,
we have that $\widetilde{\gamma}_{1}=-\gamma_{0}>-\gamma_{1}=\widetilde{\gamma}_{0}$,
$\widetilde{s}_{1}=-s_{0}\ge-s_{1}=\widetilde{s}_{0}$ and 
\begin{equation}
[\mathcal{H}_{\widetilde{p}}^{\widetilde{s}_{0},\widetilde{\gamma}_{0}}(\mathbb{B}),\mathcal{H}_{\widetilde{p}}^{\widetilde{s}_{1},\widetilde{\gamma}_{1}}(\mathbb{B})]_{\widetilde{\theta}}\hookrightarrow\mathcal{H}_{\widetilde{p}}^{\widetilde{s},\widetilde{\gamma}}(\mathbb{B}),\label{rightem}
\end{equation}
for $\widetilde{s}=(1-\widetilde{\theta})\widetilde{s_0}+\widetilde{\theta}\widetilde{s_1}$ and $\widetilde{\gamma}=(1-\widetilde{\theta})\widetilde{\gamma_0}+\widetilde{\theta}\widetilde{\gamma_1}$.

As $\widetilde{p}$, $\widetilde{s}_{0}$, $\widetilde{s}_{1}$, $\widetilde{\gamma}_{0}$
and $\widetilde{\gamma}_{1}$, with $\widetilde{\gamma}_{1}>\widetilde{\gamma}_{0}$
and $\widetilde{s}_{1}\ge\widetilde{s}_{0}$, are arbitrary, the embeddings 
\eqref{leftemb} and \eqref{rightem}
complete the proof.
\end{proof}

\subsection{Realizations of the Laplacian on conic manifolds}\label{realizations}

We focus now on the Laplacian $\Delta$ and regard it as an unbounded operator in $\mathcal{H}^{s,\gamma}_p(\mathbb{B})$, $p\in(1,\infty)$, $s,\gamma\in\mathbb{R}$, with domain $C_{c}^{\infty}(\mathbb{B}^{\circ})$. The domain of its minimal extension (i.e. its closure) $\underline{\Delta}_{s,\min}$ is given by 
$$
\mathcal{D}(\underline{\Delta}_{s,\min})=\Big\{u\in \bigcap_{\varepsilon>0}\mathcal{H}^{s+2,\gamma+2-\varepsilon}_p(\mathbb{B}) \, |\, \Delta u\in \mathcal{H}^{s,\gamma}_p(\mathbb{B})\Big\}.
$$ 
In particular
$$
\mathcal{H}_{p}^{s+2,\gamma+2}(\mathbb{B})\hookrightarrow \mathcal{D}(\underline{\Delta}_{s,\min}) \hookrightarrow \mathcal{H}_{p}^{s+2,\gamma+2-\varepsilon}(\mathbb{B})
$$
for all $\varepsilon>0$, and if
\begin{gather}\label{Qset}
Q_{\Delta}\cap \Big\{z\in\mathbb{C}\, |\, \mathrm{Re}(z)=\frac{n-3}{2}-\gamma\Big\}=\emptyset
\end{gather}
then 
$$
\mathcal{D}(\underline{\Delta}_{s,\min}) = \mathcal{H}_{p}^{s+2,\gamma+2}(\mathbb{B}),
$$
where
$$
Q_{\Delta}=\bigcup_{\lambda_{j}\in\sigma(\Delta_{h(0)})}\Big\{\frac{n-1}{2}\pm\sqrt{\Big(\frac{n-1}{2}\Big)^{2}-\lambda_{j}}\Big\}.
$$
The set of points $Q_{\Delta}$ coincides with the poles of the inverse of the family of differential operators
$$
\mathbb{C}\ni\lambda\mapsto\lambda^{2}-(n-1)\lambda+\Delta_{h(0)}\in\mathcal{L}(H_{2}^{2}(\partial\mathbb{B}),H_{2}^{0}(\partial\mathbb{B})),
$$
which is called {\em conormal symbol} of $\Delta$.

The domain of the maximal extension $\underline{\Delta}_{s,\max}$ of $\Delta$, defined by
\begin{gather*}
\mathcal{D}(\underline{\Delta}_{s,\max})=\Big\{u\in\mathcal{H}^{s,\gamma}_p(\mathbb{B}) \, |\, \Delta u\in \mathcal{H}^{s,\gamma}_p(\mathbb{B})\Big\},
\end{gather*}
is expressed as
\begin{gather}\label{dmax1}
\mathcal{D}(\underline{\Delta}_{s,\max})=\mathcal{D}(\underline{\Delta}_{s,\min})\oplus\mathcal{E}_{\Delta,\gamma}.
\end{gather}
Here $\mathcal{E}_{\Delta,\gamma}$ is a finite-dimensional $s$-independent space that is called {\em asymptotics space} and consists of linear combinations of $C^{\infty}(\mathbb{B}^\circ)$ functions that vanish on $\mathcal{B}\backslash([0,1)\times\partial\mathcal{B})$ and in local coordinates $(x,y)\in (0,1)\times\partial\mathcal{B}$ on the collar part they are of the form $\omega(x)c(y)x^{-\rho}\log^{k}(x)$, where $c\in C^{\infty}(\partial\mathbb{B})$, $\rho\in\{z\in\mathbb{C}\, |\, \mathrm{Re}(z)\in [\frac{n-3}{2}-\gamma,\frac{n+1}{2}-\gamma)\}$ and $k\in\{0,1\}$. The exponents $\rho$ are determined explicitly by the metric $h(\cdot)$.

We note that, according to \eqref{mellinsobolevinteger}, functions of the form $\omega(x)c(y)x^{-\rho}\log^{k}(x)$ belong to the space $\cap_{\varepsilon>0}\mathcal{H}_{p}^{\infty,\frac{n+1}{2}-\rho-\varepsilon}(\mathbb{B})$, which does not depend on $p$, due to Lemma \ref{propms} (vi). However, unless $c\equiv0$, they do not belong to $\mathcal{H}_{p}^{\infty,\frac{n+1}{2}-\rho}(\mathbb{B})$, for any $p\in(1,\infty)$. In particular, nonzero elements of $\mathcal{E}_{\Delta,\gamma}$ belong to $\mathcal{H}_{p}^{\infty,\gamma}(\mathbb{B})\backslash\mathcal{H}_{p}^{\infty,\gamma+2}(\mathbb{B})$.

Due to \eqref{dmax1}, there are several closed extensions of $\Delta$ in $\mathcal{H}^{s,\gamma}_p(\mathbb{B})$, each one corresponding to a subspace of $\mathcal{E}_{\Delta,\gamma}$. For an overview on the domain structure of a general $\mathbb{B}$-elliptic cone differential operator we refer to \cite[Section 3]{GKM} or alternatively to \cite[Sections 2.2 and 2.3]{SS}.

Note that the maximal and minimal domain also depend on the choice of $p\in(1,\infty)$ and $\gamma\in\mathbb{R}$. When this is relevant, we use the notation $\mathcal{D}(\underline{\Delta}_{s,p,\min})$ or $\mathcal{D}(\underline{\Delta}_{s,p,\gamma,\min})$ to denote the minimal domain and similar notations for other domains.

For a suitable choice of the space of functions and domains for the Laplacian and bi-Laplacian operator, we make first some very basic physical considerations. Usually, the solution of the Cahn-Hilliard equation provides a function that describes the concentration of the phases in the process of phase separation of cooling binary solutions. For this reason:
\begin{itemize}
\item It would be reasonable for the solution for positive times to be bounded, and possibly smooth;
\item Constant solutions such as $u=1$ or $u=-1$ should be allowed, as they represent pure phases in the process;
\item The problem should be well-posed, that is, the operator we choose must be the generator of an analytic semigroup.
\end{itemize}
A possible choice for the domain of $\Delta$, which we will see that is physically reasonable, is given by the space $X_{1}^{s}\subset X_{0}^{s}=\mathcal{H}_{p}^{s,\gamma}(\mathbb{B})$ defined below, where $\gamma$ can be smaller than $0$ when $n<3$, allowing even more singular spaces than $\mathcal{H}_{p}^{0,0}(\mathbb{B})$ in these cases.

\begin{definition}\label{constfunt}
Recall that $\partial\mathcal{B}=\cup_{j=1}^{k_{\mathcal{B}}}\partial\mathcal{B}_{j}$, for a certain $k_{\mathcal{B}}\in\mathbb{N}\backslash\{0\}$, where $\partial\mathcal{B}_{j}$ are closed, smooth and connected. Denote by $\mathbb{C}_{\omega}$ the space of all $C^{\infty}(\mathbb{B}^\circ)$ functions $c$ that vanish on $\mathcal{B}\backslash([0,1)\times\partial\mathcal{B})$ and on each component $[0,1)\times\partial\mathcal{B}_{j}$, $j\in\{1,...,k_{\mathbb{B}}\}$, they are of the form $c_{j}\omega$, where $c_{j}\in\mathbb{C}$, i.e. $\mathbb{C}_{\omega}$ consists of smooth functions that are locally constant close to the boundary. Endow $\mathbb{C}_{\omega}$ with the norm $\|\cdot\|_{\mathbb{C}_{\omega}}$ given by $c\mapsto \|c\|_{\mathbb{C}_{\omega}}=(\sum_{j=1}^{k_{\mathcal{B}}}|c_{j}|^{2})^{\frac{1}{2}}$. 
\end{definition}

If we restrict the weight $\gamma$ to the interval $(\frac{n-3}{2},\frac{n+1}{2})$, then $0\in Q_{\Delta}$ is always contained in the strip $\{z\in\mathbb{C}\, |\, \mathrm{Re}(z)\in (\frac{n-3}{2}-\gamma,\frac{n+1}{2}-\gamma)\}$. This allows us the following choice. 

\begin{domainL*}
Let $p\in(1,\infty)$, $s\geq0$ and $\gamma$ be as in \eqref{gamma}.
Then the map
\begin{gather}\label{ddelta}
\Delta:X_{1}^{s}=\mathcal{D}(\underline{\Delta}_{s})=\mathcal{D}(\underline{\Delta}_{s,\min})\oplus\mathbb{C}_{\omega}=\mathcal{H}_{p}^{s+2,\gamma+2}(\mathbb{B})\oplus\mathbb{C}_{\omega}\rightarrow X_{0}^{s}=\mathcal{H}_{p}^{s,\gamma}(\mathbb{B})
\end{gather}
defines a closed extension of $\Delta$ in $\mathcal{H}_{p}^{s,\gamma}(\mathbb{B})$ which we denote by $\underline{\Delta}_{s}$.
\end{domainL*}

\begin{remark}\label{Bancalg}
Let $p\in(1,\infty)$, $s+2>\frac{n+1}{p}$ and $\gamma+2\ge\frac{n+1}{2}$. Then, by \cite[Lemma 3.2]{RS3} we have that $X_{1}^{s}$ is a Banach algebra up to an equivalent norm. Moreover, $X_{1}^{s}\hookrightarrow C(\mathbb{B})$.
\end{remark}

The description of the domain of the bi-Laplacian is now straightforward.

\begin{domain*}
Let $p\in(1,\infty)$, $s\geq0$ and $\gamma$ be as in \eqref{gamma}. The domain of the bi-Laplacian $\underline{\Delta}_{s}^{2}$ associated to the Laplacian $\underline{\Delta}_{s}$ from \eqref{ddelta}, defined as usual by 
\begin{gather*}\label{pppo}
\mathcal{D}(\underline{\Delta}_{s}^{2})=\{u\in\mathcal{D}(\underline{\Delta}_{s})\, |\, \underline{\Delta}_{s}u\in \mathcal{D}(\underline{\Delta}_{s})\},
\end{gather*}
has the following form
\begin{gather}\label{bilapldmn}
X_{2}^{s}=\mathcal{D}(\underline{\Delta}_{s,\min}^{2})\oplus \mathbb{C}_{\omega}\oplus \mathcal{E}_{\Delta^{2},\gamma}.
\end{gather}

The minimal domain satisfies
$$
\mathcal{H}_{p}^{s+4,\gamma+4}(\mathbb{B})\hookrightarrow \mathcal{D}(\underline{\Delta}_{s,\min}^{2}) \hookrightarrow \mathcal{H}_{p}^{s+4,\gamma+4-\varepsilon}(\mathbb{B})
$$
for all $\varepsilon>0$, and if
$$
Q_{\Delta^{2}}\cap \Big\{z\in\mathbb{C}\, |\, \mathrm{Re}(z)=\frac{n-7}{2}-\gamma\Big\}=\emptyset
$$
then 
$$
\mathcal{D}(\underline{\Delta}_{s,\min}^{2}) = \mathcal{H}_{p}^{s+4,\gamma+4}(\mathbb{B}),
$$
where
$$
Q_{\Delta^{2}}=\bigcup_{\lambda_{j}\in\sigma(\Delta_{h(0)})}\Big\{\frac{n-1}{2}\pm\sqrt{\Big(\frac{n-1}{2}\Big)^{2}-\lambda_{j}}\Big\}\cup\Big\{\frac{n-5}{2}\pm\sqrt{\Big(\frac{n-1}{2}\Big)^{2}-\lambda_{j}}\Big\}.
$$

Moreover, $\mathcal{E}_{\Delta^{2},\gamma}$ is a finite dimensional $s$-independent space consisting of linear combinations of $C^{\infty}(\mathbb{B}^{\circ})$-functions that are equal to zero on $\mathcal{B}\backslash([0,1)\times\partial\mathcal{B})$ and in local coordinates $(x,y)\in (0,1)\times\partial\mathcal{B}$ on the collar part they are of the form $\omega(x)c(y)x^{-\rho}\log^{k}(x)$, where $c\in C^{\infty}(\partial\mathbb{B})$, $\rho\in\{z\in\mathbb{C}\, |\, \mathrm{Re}(z)\in [\frac{n-7}{2}-\gamma,\frac{n-3}{2}-\gamma)\}$ and $k\in\{0,1,2,3\}$. In particular, there exists a $0<\delta_{0}<2$ such that 
\begin{gather}\label{asympembedd}
\mathcal{E}_{\Delta^{2},\gamma}\hookrightarrow \mathcal{H}_{p}^{\infty,\gamma+2+\delta_{0}}(\mathbb{B}).
\end{gather}
The constant $\delta_{0}$ is determined by the fact that it should be smaller than $\frac{n+1}{2}-(\rho+\gamma+2)$ for all $\rho$ such that $\omega(x)c(y)x^{-\rho}\log^{k}(x)$ belongs to $\mathcal{E}_{\Delta^{2},\gamma}$. Once we fix such a constant $\delta_{0}$, we can always find a slightly large constant $\widetilde{\delta}_{0}$ with the same properties.

Finally, if $h(x)$ is constant in $x$ when $x$ is close to $0$, then 
$$
\mathcal{E}_{\Delta^{2},\gamma}=\bigoplus_{\rho\in Q_{\Delta^{2},\gamma}\cap [\frac{n-7}{2}-\gamma,\frac{n-3}{2}-\gamma)} \mathcal{E}_{\Delta^{2},\gamma,\rho}.
$$
Here each $\mathcal{E}_{\Delta^{2},\gamma,\rho}$ is a finite dimensional $s$-independent space consisting of linear combinations of $C^{\infty}(\mathbb{B}^{\circ})$-functions that are equal to zero on $\mathcal{B}\backslash([0,1)\times\partial\mathcal{B})$ and in local coordinates $(x,y)\in (0,1)\times\partial\mathcal{B}$ on the collar part they are of the form $\omega(x)c(y)x^{-\rho}\log^{k}(x)$, where $c\in C^{\infty}(\partial\mathbb{B})$ and $k\in\{0,1,2,3\}$.
\end{domain*}

By the restriction of the weight $\gamma$ in terms of the spectrum of the cross-section Laplacian $\Delta_{h(0)}$ as in \eqref{gamma}, we have that $X_{1}^{0}$ in the case of $p=2$ is contained in a domain of a self-adjoint extension of $\Delta$ in $\mathcal{H}_{2}^{0,0}(\mathbb{B})$, see the proof of \cite[Theorem 5.7]{SS}. Then, $\underline{\Delta}_{s}$ fulfills certain ellipticity conditions with respect to an arbitrary sector in the complex plane, see (E1)-(E3) in \cite[Section 3.2]{SS} or (E1)-(E3) in \cite[Section 4]{SS1}. In this situation the above described Laplacian and bi-Laplacian satisfy the properties of sectoriality and boundedness of the imaginary powers as we can see from the following. 

\begin{proposition}\label{bipforbi}
Let $p\in(1,\infty)$, $s\geq0$, $\gamma$ be as in \eqref{gamma}, $\underline{\Delta}_{s}$ be the Laplacian \eqref{ddelta} and $c>0$. Then, for each $\theta\in[0,\pi)$ and $\phi>0$ we have that both $c-\underline{\Delta}_{s}$ and $(c-\underline{\Delta}_{s})^{2}$ belong to $\mathcal{P}(\theta)\cap\mathcal{BIP}(\phi)$.
\end{proposition}
\begin{proof}
Concerning $c-\underline{\Delta}_{s}$, the sectoriality follows by \cite[Theorem 4.1]{Ro1} or \cite[Theorem 5.6]{RS2}. The boundedness of the imaginary powers can be obtained from \cite[Theorem 6.7]{SS1}. In the case of straight conical tips, i.e. when $h(\cdot)$ is constant close to the boundary, the result alternatively follows by \cite[Theorems 5.6 and 5.7]{SS}, see also \cite[Section 3.1]{RS1}.

Concerning $(c-\underline{\Delta}_{s})^{2}$, the result follows from \cite[Lemma 3.6]{RS1}.
\end{proof}

\begin{proposition}\label{biuytpfohit}
Let $p\in(1,\infty)$, $s\geq0$, $\gamma$ be as in \eqref{gamma} and $\underline{\Delta}_{s}$ be the Laplacian \eqref{ddelta}. Then, for any $\theta\in[0,\pi)$ and $\phi>0$, there exists a $c_{0}>0$ such that $\underline{\Delta}_{s}^{2}+c_{0}\in\mathcal{P}(\theta)\cap\mathcal{BIP}(\phi)$.
\end{proposition}
\begin{proof}
We apply the perturbation result \cite[Theorem III.4.8.5]{Am} with $A=(c-\underline{\Delta}_{s})^{2}+\eta$ and $B=-2c\underline{\Delta}_{s}$, $c,\eta>0$. By Proposition \ref{bipforbi} and \cite[Corollary III.4.8.6]{Am} we have that $A\in\mathcal{P}(\theta)\cap\mathcal{BIP}(\phi)$. Moreover, $\mathcal{D}(A)\subset\mathcal{D}(B)$. Denote by $K_{(c-\underline{\Delta}_{s})^{2}+\eta,\theta}$ the sectorial bound of $(c-\underline{\Delta}_{s})^{2}+\eta$ with respect to angle $\theta$. Note that $K_{(c-\underline{\Delta}_{s})^{2}+\eta,\theta}$ is uniformly bounded in $\eta\geq0$ (see e.g. \cite[Lemma 2.6]{RS3}). For each $\lambda\in S_{\theta}$ and $\rho\in (0,\frac{1}{2})$ we have
\begin{eqnarray*}
\lefteqn{\|B(A+\lambda)^{-1}\|_{\mathcal{L}(X_{0}^{s})}}\\
&\leq&2c\|\underline{\Delta}_{s}(c-\underline{\Delta}_{s})^{-1}\|_{\mathcal{L}(X_{0}^{s})}\\
&&\times\|(c-\underline{\Delta}_{s})((c-\underline{\Delta}_{s})^{2}+\eta)^{-1}\|_{\mathcal{L}(X_{0}^{s})}\|((c-\underline{\Delta}_{s})^{2}+\eta)((c-\underline{\Delta}_{s})^{2}+\eta+\lambda)^{-1}\|_{\mathcal{L}(X_{0}^{s})}\\
&\leq&2c\frac{C_{1}}{1+\eta^{\rho}}K_{(c-\underline{\Delta}_{s})^{2}+\eta,\theta}\|\underline{\Delta}_{s}(c-\underline{\Delta}_{s})^{-1}\|_{\mathcal{L}(X_{0}^{s})},
\end{eqnarray*}
where we have used \cite[Lemma 2.4]{Ro3}. The above constant $C_{1}>0$ depends only on $\rho$ and the sectorial bound of $(c-\underline{\Delta}_{s})^{2}$. Hence, by taking $\eta$ sufficiently large we can make $\|B(A+\lambda)^{-1}\|_{\mathcal{L}(X_{0}^{s})}<\sigma<1$ uniformly in $\lambda\in S_{\theta}$.

Furthermore, we estimate
 \begin{eqnarray*}
\lefteqn{\|(A+\lambda)^{-1}B(A+\lambda)^{-1}\|_{\mathcal{L}(X_{0}^{s})}}\\
&\leq&2c\|((c-\underline{\Delta}_{s})^{2}+\eta+\lambda)^{-1}\|_{\mathcal{L}(X_{0}^{s})}\|\underline{\Delta}_{s}(c-\underline{\Delta}_{s})^{-1}\|_{\mathcal{L}(X_{0}^{s})}\|(c-\underline{\Delta}_{s})((c-\underline{\Delta}_{s})^{2}+\eta+\lambda)^{-1}\|_{\mathcal{L}(X_{0}^{s})}\\
&\leq&2c\frac{K_{(c-\underline{\Delta}_{s})^{2}+\eta,\theta}}{1+|\lambda|}\frac{C_{2}}{1+|\eta+\lambda|^{\rho}}\|\underline{\Delta}_{s}(c-\underline{\Delta}_{s})^{-1}\|_{\mathcal{L}(X_{0}^{s})}
\end{eqnarray*}
for some $C_{2}>0$ depending on $\rho$ and the sectorial bound of $(c-\underline{\Delta}_{s})^{2}$, where we have used \cite[Lemma 2.4]{Ro3} once more. Therefore, the assumptions (i), (ii) and (iii) of \cite[Theorem III.4.8.5]{Am} are satisfied and the result follows. 
\end{proof}

Notice that the above choice of domains for the Laplacian and bi-Laplacian contains constant functions. Moreover, the operators are generators of analytic semigroups and, if the solutions of \eqref{CH1}-\eqref{CH2} regularize to a space at least as good as $X_{1}^{s}$ for $s+2>\frac{n+1}{p}$, then they are bounded for any positive time. Therefore, they are appropriate choices of domain according to our earlier physical considerations.

\section{The Cahn-Hilliard equation on conic manifolds}

\subsection{Short time solution and smoothness}

In order to apply the results of Section 2 to our concrete situation, we need to have some immersion of the interpolation spaces between the domain and the underline space.

\begin{lemma}\label{interfind} Let $p\in(1,\infty)$, $s\geq0$, $\gamma$ be as
in \eqref{gamma} and $\delta_{0}$ be as in \eqref{asympembedd}. \\
{\em (i)} If $q>2$, then
$$
(\mathcal{D}(\underline{\Delta}_{s}^{2}),\mathcal{H}_{p}^{s,\gamma}(\mathbb{B}))_{\frac{1}{q},q}\hookrightarrow\bigcap_{\epsilon>0}\mathcal{H}_{p}^{s+4(1-\frac{1}{q})-\epsilon,\gamma+2+\delta_{0}(1-\frac{2}{q})}(\mathbb{B})\oplus\mathbb{C}_{\omega}.
$$
{\em (ii)} If $p=q=2$, then
$$
(\mathcal{D}(\underline{\Delta}_{s}^{2}),\mathcal{H}_{2}^{s,\gamma}(\mathbb{B}))_{\frac{1}{2},2}=\mathcal{H}_{2}^{s+2,\gamma+2}(\mathbb{B})\oplus\mathbb{C}_{\omega}
$$
up to norm equivalence.
\end{lemma}

\begin{proof}
(i) By \cite[(I.2.5.4)]{Am} and \cite[Corollary
7.3 (c)]{Haa} we have that 
$$
(\mathcal{D}(\underline{\Delta}_{s}^{2}),\mathcal{H}_{p}^{s,\gamma}(\mathbb{B}))_{\frac{1}{q},q}=(\mathcal{D}(\underline{\Delta}_{s}^{2}),\mathcal{D}(\underline{\Delta}_{s}))_{\frac{2}{q},q}\hookrightarrow(\mathcal{H}_{p}^{s+4,\gamma+2+\delta_{0}}(\mathbb{B})\oplus\mathbb{C}_{\omega},\mathcal{H}_{p}^{s+2,\gamma+2}(\mathbb{B})\oplus\mathbb{C}_{\omega})_{\frac{2}{q},q}=E,
$$
 where we have used the description of the bi-Laplacian domain; in particular \eqref{asympembedd}. We follow the ideas in \cite[Theorem 3.3]{Ro2}. By \cite[(I.2.5.2)]{Am} write any $v\in E$ as $v=w+a$, where $w\in \mathcal{H}_{p}^{s+2,\gamma+2}(\mathbb{B})$ and $a\in \mathbb{C}_{\omega}$. Due to \cite[Theorem B.2.3]{Haa2} we have that $\Delta$ maps $E$ to 
\begin{equation*}\label{eq:Emapsto}
(\mathcal{H}_{p}^{s+2,\gamma+\delta_{0}}(\mathbb{B}),\mathcal{H}_{p}^{s,\gamma}(\mathbb{B}))_{\frac{2}{q},q}\hookrightarrow\mathcal{H}_{p}^{s+2(1-\frac{2}{q})-\varepsilon,\gamma+\delta_{0}(1-\frac{2}{q})-\varepsilon}(\mathbb{B})
\end{equation*}
 for all $\varepsilon>0$, where we have used Lemma \ref{prop:interpolation} (i). Hence $w$ belongs to the maximal domain of $\Delta$ in $\mathcal{H}_{p}^{s+2(1-\frac{2}{q})-\varepsilon,\gamma+\delta_{0}(1-\frac{2}{q})-\varepsilon}(\mathbb{B})$.
We conclude that
\begin{gather}
E\hookrightarrow\mathcal{H}_{p}^{s+4(1-\frac{1}{q})-\epsilon,\gamma+2+\delta_{0}(1-\frac{2}{q})-\varepsilon}(\mathbb{B})\oplus\mathcal{E}_{\Delta,\gamma+\delta_{0}(1-\frac{2}{q})-\varepsilon}\oplus\mathbb{C}_{\omega},\label{embdkah}
\end{gather}
 for all $\varepsilon>0$. Note that $\mathcal{E}_{\Delta,\gamma+\delta_{0}(1-\frac{2}{q})-\varepsilon}\oplus\mathbb{C}_{\omega}=\mathcal{E}_{\Delta,\gamma+\delta_{0}(1-\frac{2}{q})-\varepsilon}$
if $\gamma+\delta_{0}(1-\frac{2}{q})-\varepsilon<\frac{n+1}{2}$.
Furthermore, $\omega(x)x\partial_{x}$ maps $E$ to 
$$
(\mathcal{H}_{p}^{s+3,\gamma+2+\delta_{0}}(\mathbb{B}),\mathcal{H}_{p}^{s+1,\gamma+2}(\mathbb{B}))_{\frac{2}{q},q}\hookrightarrow\mathcal{H}_{p}^{s+1+2(1-\frac{2}{q})-\varepsilon,\gamma+2+\delta_{0}(1-\frac{2}{q})-\varepsilon}(\mathbb{B}),
$$
 for all $\epsilon>0$. By the above embedding, \eqref{embdkah},
the structure of $\mathcal{E}_{\Delta,\gamma+\delta_{0}(1-\frac{2}{q})-\varepsilon}\oplus\mathbb{C}_{\omega}$
and the identity $x\partial_{x}(x^{m}\log^{k}(x))=mx^{m}\log^{k}(x)+kx^{m}\log^{k-1}(x)$, $m\in\mathbb{C}$, $k\in\mathbb{N}$,
we see that $\mathcal{E}_{\Delta,\gamma+\delta_{0}(1-\frac{2}{q})-\varepsilon}\oplus\mathbb{C}_{\omega}$
can consist only of linear combinations of $\mathbb{C}_{\omega}$
terms.

As the above argument holds for all $\varepsilon>0$ and as we can
always find a slightly bigger $\delta_{0}$ with the same properties
as in \eqref{asympembedd}, we conclude our proof.

(ii) We have
$$
(\mathcal{D}(\underline{\Delta}_{s}^{2}),\mathcal{H}_{2}^{s,\gamma}(\mathbb{B}))_{\frac{1}{2},2}=(\mathcal{H}_{2}^{s,\gamma}(\mathbb{B}),\mathcal{D}(\underline{\Delta}_{s}^{2}))_{\frac{1}{2},2}=[\mathcal{H}_{2}^{s,\gamma}(\mathbb{B}),\mathcal{D}(\underline{\Delta}_{s}^{2})]_{\frac{1}{2}}=\mathcal{D}(\underline{\Delta}_{s}),
$$
where we have used \cite[(I.2.5.4)]{Am} in the first equality, \cite[Corollary 4.37]{Lunardi} in the second equality and the fact that $\underline{\Delta}_{s}^{2}+c_{0}$, $c_{0}>0$, has bounded imaginary powers in the third equality, i.e. \cite[(I.2.9.8)]{Am}.
\end{proof}
We can now apply our previous results to the study of the Cahn-Hilliard equation.

\subsubsection*{Proof of Theorem \ref{ThLTS}, first part (short time existence and regularity).}

The immersion of \eqref{u0reg} follows from the fact that $\mathbb{C}_{\omega}\subset\mathcal{D}(\underline{\Delta}_{s}^{2})\cap\mathcal{H}_{p}^{s,\gamma}(\mathbb{B})$ and from Lemma \ref{prop:interpolation} (ii).

The proof of the existence and uniqueness of a solution satisfying
\eqref{uexists} is a simple application of Theorem
\ref{thcl} to \eqref{CH1}-\eqref{CH2}
with $X_{0}=\mathcal{H}_{p}^{s,\gamma}(\mathbb{B})$, $X_{1}=\mathcal{D}(\underline{\Delta}_{s}^{2})$,
$U$ an open bounded subset of $(X_{1},X_{0})_{\frac{1}{q},q}$ containing $u_{0}$, $A=\underline{\Delta}_{s}^{2}$
and $F(\cdot)=\underline{\Delta}_{s}((\cdot)^{3}-(\cdot))$.

In fact, the operator $A$ has maximal $L^{q}$-regularity due to
Proposition \ref{biuytpfohit} and Theorem \ref{dorevenni}.
The function $F:(\mathcal{D}(\underline{\Delta}_{s}^{2}),\mathcal{H}_{p}^{s,\gamma}(\mathbb{B}))_{\frac{1}{q},q}\to\mathcal{H}_{p}^{s,\gamma}(\mathbb{B})$
given by $F(u)=\underline{\Delta}_{s}(u^{3}-u)$
is $C^{\infty}$ and, in particular, locally Lipschitz, as we will see later on. The property \eqref{contsolu}
follows by Remark \ref{embedtoC}, Lemma \ref{interfind}
and Remark \ref{Bancalg}. 

For \eqref{extrareg}, we apply Theorem \ref{abstractshortsmooth}
to the Banach scale $Y_{0}^{k}=X_{0}^{s+k}$, $Y_{1}^{k}=X_{2}^{s+k}$,
$k\in\mathbb{N}$. We will suppose that $q<4$, without loss of generality.
In fact, if $q\ge4$, we can choose $\widetilde{q}\in(2,4)$
and note that $(X_{1},X_{0})_{\frac{1}{q},q}\hookrightarrow(X_{1},X_{0})_{\frac{1}{\widetilde{q}},\widetilde{q}}$
and 
$$
W^{1,q}(0,T;\mathcal{H}_{p}^{s,\gamma}(\mathbb{B}))\cap L^{q}(0,T;\mathcal{D}(\underline{\Delta}_{s}^{2}))\hookrightarrow W^{1,\widetilde{q}}(0,T;\mathcal{H}_{p}^{s,\gamma}(\mathbb{B}))\cap L^{\widetilde{q}}(0,T;\mathcal{D}(\underline{\Delta}_{s}^{2})).
$$
Hence the smoothness will follow from the case where $q<4$.

It remains to check the assumptions of Theorem \ref{abstractshortsmooth}.
It is clear by the definition that $Y_{1}^{k}\overset{d}{\hookrightarrow}Y_{0}^{k}$,
$Y_{0}^{k+1}\overset{d}{\hookrightarrow}Y_{0}^{k}$ and $Y_{1}^{k+1}\overset{d}{\hookrightarrow}Y_{1}^{k}$.
In order to prove that $Y_{1}^{k}\overset{d}{\hookrightarrow}(Y_{1}^{k+1},Y_{0}^{k+1})_{\frac{1}{q},q}$,
we first note that, by \eqref{bilapldmn}, we have
\begin{gather}
Y_{1}^{k}=\mathcal{D}(\underline{\Delta}_{s+k,\min}^{2})\oplus\mathbb{C}_{\omega}\oplus\mathcal{E}_{\Delta^{2},\gamma}\hookrightarrow\mathcal{H}_{p}^{s+k+4,\gamma+4-\varepsilon}(\mathbb{B})\oplus\mathbb{C}_{\omega}\oplus\mathcal{E}_{\Delta^{2},\gamma}\label{pkadya}
\end{gather}
 for all $\varepsilon>0$, where the finite dimensional space $\mathcal{E}_{\Delta^{2},\gamma}$
is independent of the order $s+k$. As $q<4$, Lemma \ref{prop:interpolation} (ii) implies that
$$
\mathcal{H}_{p}^{s+k+4,\gamma+4-\varepsilon}(\mathbb{B})\hookrightarrow(\mathcal{H}_{p}^{s+(k+1)+4,\gamma+4}(\mathbb{B}),\mathcal{H}_{p}^{s+(k+1),\gamma}(\mathbb{B}))_{\frac{1}{q},q}
$$
 for all $\varepsilon>0$ sufficiently small. Therefore for any such
$\varepsilon>0$
$$
\mathcal{H}_{p}^{s+k+4,\gamma+4-\varepsilon}(\mathbb{B})\hookrightarrow(\mathcal{H}_{p}^{s+(k+1)+4,\gamma+4}(\mathbb{B})\oplus\mathbb{C}_{\omega}\oplus\mathcal{E}_{\Delta^{2},\gamma},\mathcal{H}_{p}^{s+(k+1),\gamma}(\mathbb{B}))_{\frac{1}{q},q}.
$$

By noting that 
$$
\mathbb{C}_{\omega}\oplus\mathcal{E}_{\Delta^{2}}\hookrightarrow(\mathcal{H}_{p}^{s+(k+1)+4,\gamma+4}(\mathbb{B})\oplus\mathbb{C}_{\omega}\oplus\mathcal{E}_{\Delta^{2},\gamma},\mathcal{H}_{p}^{s+(k+1),\gamma}(\mathbb{B}))_{\frac{1}{q},q},
$$
\eqref{pkadya} implies
$$
Y_{1}^{k}\hookrightarrow(Y_{1}^{k+1},Y_{0}^{k+1})_{\frac{1}{q},q}.
$$

The assumption (i) of Theorem \ref{abstractshortsmooth} is a direct consequence of the definition of the spaces $Y_{i}^{k}$, Proposition \ref{biuytpfohit} and Theorem \ref{dorevenni}.

For the assumption (ii) of Theorem \ref{abstractshortsmooth}, we check first that $F:(Y_{1}^{k},Y_{0}^{k})_{\frac{1}{q},q}\to Y_{0}^{k}$ is smooth. If $q>2$, Lemma \ref{interfind} implies that, for $0<\epsilon<2-\frac{4}{q}$,
$$
(Y_{1}^{k},Y_{0}^{k})_{\frac{1}{q},q}\hookrightarrow\mathcal{H}_{p}^{s+k+4(1-\frac{1}{q})-\epsilon,\gamma+2+\delta_{0}(1-\frac{2}{q})}(\mathbb{B})\oplus\mathbb{C}_{\omega}.
$$

As the space $\mathcal{H}_{p}^{s+k+4(1-\frac{1}{q})-\epsilon,\gamma+2+\delta_{0}(1-\frac{2}{q})}(\mathbb{B})\oplus\mathbb{C}_{\omega}$
is a Banach algebra, due to Remark \ref{Bancalg},
we conclude that, if $v\in(Y_{1}^{k},Y_{0}^{k})_{\frac{1}{q},q}$,
then $v,v^{3}\in\mathcal{H}_{p}^{s+k+4(1-\frac{1}{q})-\epsilon,\gamma+2+\delta_{0}(1-\frac{2}{q})}(\mathbb{B})\oplus\mathbb{C}_{\omega}.$
Therefore, $F(v)\in\mathcal{H}_{p}^{s+k+2-\frac{4}{q}-\epsilon,\gamma+\delta_{0}(1-\frac{2}{q})}(\mathbb{B})\hookrightarrow Y_{0}^{k}$
as $\epsilon<2-\frac{4}{q}$. It is also clear that $F:(Y_{1}^{k},Y_{0}^{k})_{\frac{1}{q},q}\to Y_{0}^{k}$
is a $C^{\infty}$-function, as it is the composition $F=F_{3}\circ F_{2}\circ F_{1}$,
where 
$$
\begin{array}{lcl}
F_{1}:(Y_{1}^{k},Y_{0}^{k})_{\frac{1}{q},q}\to\mathcal{H}_{p}^{s+k+4(1-\frac{1}{q})-\epsilon,\gamma+2+\delta_{0}(1-\frac{2}{q})}(\mathbb{B})\oplus\mathbb{C}_{\omega}, & & F_{1}(u)=u,\\
F_{2}:\mathcal{H}_{p}^{s+k+4(1-\frac{1}{q})-\epsilon,\gamma+2+\delta_{0}(1-\frac{2}{q})}(\mathbb{B})\oplus\mathbb{C}_{\omega}\circlearrowleft, & & F_{2}(u)=u^{3}-u,\\
F_{3}:\mathcal{H}_{p}^{s+k+4(1-\frac{1}{q})-\epsilon,\gamma+2+\delta_{0}(1-\frac{2}{q})}(\mathbb{B})\oplus\mathbb{C}_{\omega}\to Y_{0}^{k}, & & \,F_{3}(u)=\underline{\Delta}_{s}u,
\end{array}
$$
 where $\circlearrowleft$ indicates a function with the same domain
and codomain. The functions $F_{1}$ and $F_{3}$ are linear, so they
are smooth as well. The function $F_{2}$ is smooth due the fact that, according to Remark \ref{Bancalg}, the space
$\mathcal{H}_{p}^{s+k+4(1-\frac{1}{q})-\epsilon,\gamma+2+\delta_{0}(1-\frac{2}{q})}(\mathbb{B})\oplus\mathbb{C}_{\omega}$
is a Banach algebra.

If $p=q=2$, then Lemma \ref{interfind} implies that
$$
(Y_{1}^{k},Y_{0}^{k})_{\frac{1}{2},2}=\mathcal{H}_{2}^{s+k+2,\gamma+2}(\mathbb{B})\oplus\mathbb{C}_{\omega}
$$
and the proof follows exactly as before.

The proof that $F:Y_{1}^{k}\to Y_{0}^{k+1}$ defines a continuous
map is very similar. In fact, if 
$$
u\in Y_{1}^{k}=X_{2}^{s+k}\hookrightarrow\mathcal{H}_{p}^{s+4+k,\gamma+2+\delta_{0}}(\mathbb{B})\oplus\mathbb{C}_{\omega},
$$
where $\delta_{0}$ is as in \eqref{asympembedd},
then $u^{3}-u$ also belong to $\mathcal{H}_{p}^{s+4+k,\gamma+2+\delta_{0}}(\mathbb{B})\oplus\mathbb{C}_{\omega}$,
as this space is a Banach algebra, and 
$$
\Delta(u^{3}-u)\in\mathcal{H}_{p}^{s+2+k,\gamma+\delta_{0}}(\mathbb{B})\hookrightarrow\mathcal{H}_{p}^{s+k+1,\gamma}(\mathbb{B})\oplus\mathbb{C}=Y_{0}^{k+1}.
$$
Moreover all these operations are continuous. \mbox{\ } \hfill $\square$

It is clear that for $t>0$, \eqref{extrareg} provides a much stronger regularity than \eqref{uexists}. The interesting point of \eqref{uexists} is that it provides information of how the solution behaves in the neighborhood of the initial data, that is, as $t\to0^{+}$. Moreover, \eqref{uexists}-\eqref{extrareg} have an important implication for the proof of global existence of solutions. We will show that long time existence of solutions is, in some sense, independent of $s$, $p$ and $q$.

\begin{corollary}
\label{indmaxint} For $j\in\{1,2\}$, let $s_{j}\ge0$, $p_{j}\in(1,\infty)$, $s_{j}+2>\frac{n+1}{p_{j}}$, $\gamma$ be as in \eqref{gamma}, $q_{j}>2$ if $p_{j}\ne2$ and $q_{j}\ge2$ if $p_{j}=2$. If for every $u_{0}\in(\mathcal{D}(\underline{\Delta}_{s_{1,}p_{1}}^{2}),\mathcal{H}_{p_{1}}^{s_{1},\gamma}(\mathbb{B}))_{\frac{1}{q_{1}},q_{1}}$ the maximal interval of existence of the solution 
$$
u\in W^{1,q_{1}}(0,T;\mathcal{H}_{p_{1}}^{s_{1},\gamma}(\mathbb{B})) \cap L^{q_{1}}(0,T;\mathcal{D}(\underline{\Delta}_{s_{1},p_{1}}^{2}))
$$
of the Cahn-Hilliard equation \eqref{CH1}-\eqref{CH2} is $[0,\infty)$, then for every $v_{0}\in(\mathcal{D}(\underline{\Delta}_{s_{2},p_{2}}^{2}),\mathcal{H}_{p_{2}}^{s_{2},\gamma}(\mathbb{B}))_{\frac{1}{q_{2}},q_{2}}$
the maximal interval of existence of the solution 
$$
v\in W^{1,q_{2}}(0,T;\mathcal{H}_{p_{2}}^{s_{2},\gamma}(\mathbb{B}))\cap L^{q_{2}}(0,T;\mathcal{D}(\underline{\Delta}_{s_{2},p_{2}}^{2}))
$$
is also $[0,\infty)$.
\end{corollary}
\begin{proof}
Let $v_{0}\in(\mathcal{D}(\underline{\Delta}_{s_{2},p_{2},\gamma}^{2}),\mathcal{H}_{p_{2}}^{s_{2},\gamma}(\mathbb{B}))_{\frac{1}{q_{2}},q_{2}}$ and $v:[0,T_{\max})\to\mathcal{H}_{p_{2}}^{s_{2},\gamma}(\mathbb{B})$ be the maximal solution of \eqref{CH1}-\eqref{CH2}. By \eqref{uexists}-\eqref{extrareg}, we know that 
$$
v\in\bigcap_{s_{2}>0}C^{\infty}((0,T_{\max});\mathcal{D}(\underline{\Delta}_{s_{2},p_{2},\gamma}^{2})) \quad \text{and}\quad v|_{(0,T)}\in W^{1,q_{2}}(0,T;\mathcal{H}_{p_{2}}^{s_{2},\gamma}(\mathbb{B}))\cap L^{q_{2}}(0,T;\mathcal{D}(\underline{\Delta}_{s_{2},p_{2},\gamma}^{2}))
$$
for all $T<T_{\max}$. Let us suppose that $T_{\max}<\infty$.

As $v(T_{\max}/2)\in\cap_{s>0}\mathcal{D}(\underline{\Delta}_{s,p_{2},\gamma}^{2})$,
we have $v(T_{\max}/2)\in\mathcal{H}_{p_{2}}^{\infty,\gamma+4-\varepsilon}(\mathbb{B})\oplus\mathbb{C}_{\omega}\oplus\mathcal{E}_{\Delta^{2},\gamma}$,
for all $\varepsilon>0$. Due to Lemma \ref{propms} (vi), we know that 
$$
\bigcap_{\varepsilon>0}\mathcal{H}_{p_{1}}^{\infty,\gamma+4-\varepsilon}(\mathbb{B})=\bigcap_{\varepsilon>0}\mathcal{H}_{p_{2}}^{\infty,\gamma+4-\varepsilon}(\mathbb{B}).
$$
Hence
$$
v(T_{\max}/2)\in\bigcap_{\varepsilon>0}\mathcal{H}_{p_{1}}^{\infty,\gamma+4-\varepsilon}(\mathbb{B})\oplus\mathbb{C}_{\omega}\oplus\mathcal{E}_{\Delta^{2},\gamma}.
$$

However, 
$$
\mathcal{H}_{p_{1}}^{\infty,\gamma+4-\varepsilon}(\mathbb{B})\hookrightarrow(\mathcal{H}_{p_{1}}^{s_{1}+4,\gamma+4}(\mathbb{B}),\mathcal{H}_{p_{1}}^{s_{1},\gamma}(\mathbb{B}))_{\frac{1}{q_{1}},q_{1}}\hookrightarrow(\mathcal{D}(\underline{\Delta}_{s_{1},p_{1},\gamma}^{2}),\mathcal{H}_{p_{1}}^{s_{1},\gamma}(\mathbb{B}))_{\frac{1}{q_{1}},q_{1}},
$$
 for sufficiently small $\varepsilon$, due to Lemma \ref{prop:interpolation} (ii).
As
$$
\mathbb{C}_{\omega}\oplus\mathcal{E}_{\Delta^{2},\gamma}\hookrightarrow(\mathcal{D}(\underline{\Delta}_{s_{1},p_{1},\gamma}^{2}),\mathcal{H}_{p_{1}}^{s_{1},\gamma}(\mathbb{B}))_{\frac{1}{q_{1}},q_{1}},
$$
we conclude that $v(T_{\max}/2)\in(\mathcal{D}(\underline{\Delta}_{s_{1},p_{1},\gamma}^{2}),\mathcal{H}_{p_{1}}^{s_{1},\gamma}(\mathbb{B}))_{\frac{1}{q_{1}},q_{1}}$.

By \eqref{uexists}-\eqref{extrareg} and our assumptions about global existence of solutions for $s_{1}$, $p_{1}$ and $q_{1}$, there
is a unique solution $u\in\cap_{s_{1}\ge0}C^{\infty}((T_{\max}/2,\infty),\mathcal{D}(\underline{\Delta}_{s_{1},p_{1},\gamma}^{2}))$
of Equation \eqref{CH1} such that $u(T_{\max}/2)=v(T_{\max}/2)$
and 
$$
u\in W^{1,q_{1}}(T_{\max}/2,\widetilde{T};\mathcal{H}_{p_{1}}^{s_{1},\gamma}(\mathbb{B}))\cap L^{q_{1}}(T_{\max}/2,\widetilde{T};\mathcal{D}(\underline{\Delta}_{s_{1},p_{1},\gamma}^{2}))
$$
for all $\widetilde{T}>T_{\max}/2$.

Let us show that $v|_{(T_{\max}/2,T_{\max})}=u|_{(T_{\max}/2,T_{\max})}$.
We choose $\widetilde{\gamma}<\gamma$ such that $\widetilde{\gamma}$ also
satisfies the conditions \eqref{gamma}. For $j\in\{1,2\}$, the embeddings of Lemma \ref{propms} (iv) and (v) imply that 
$$
\mathcal{D}(\underline{\Delta}_{s_{1},p_{1},\gamma})\hookrightarrow\mathcal{D}(\underline{\Delta}_{s_{1},p_{1},\widetilde{\gamma}})\quad \text{and}\quad \bigcap_{s\ge0}\mathcal{D}(\underline{\Delta}_{s,p_{2},\gamma})\hookrightarrow\mathcal{D}(\underline{\Delta}_{s_{1},p_{1},\widetilde{\gamma}}).
$$
Therefore 
$$
\mathcal{D}(\underline{\Delta}_{s_{1},p_{1},\gamma}^{2})\hookrightarrow\mathcal{D}(\underline{\Delta}_{s_{1},p_{1},\widetilde{\gamma}}^{2})
\quad \text{and} \quad \bigcap_{s\ge0}\mathcal{D}(\underline{\Delta}_{s,p_{2},\gamma}^{2})\hookrightarrow\mathcal{D}(\underline{\Delta}_{s_{1},p_{1},\widetilde{\gamma}}^{2}).
$$
Moreover, for $T_{\max}/2<T'<T_{\max}$, we have
\begin{eqnarray*}
\lefteqn{u|_{(T_{\max}/2,T')} \in W^{1,q_{1}}(T_{\max}/2,T';\mathcal{H}_{p_{1}}^{s_{1},\gamma}(\mathbb{B}))\cap L^{q_{1}}(T_{\max}/2,T';\mathcal{D}(\underline{\Delta}_{s_{1},p_{1},\gamma}^{2}))}\\
 && \hookrightarrow W^{1,q_{1}}(T_{\max}/2,T';\mathcal{H}_{p_{1}}^{s_{1},\widetilde{\gamma}}(\mathbb{B}))\cap L^{q_{1}}(T_{\max}/2,T';\mathcal{D}(\underline{\Delta}_{s_{1},p_{1},\widetilde{\gamma}}^{2}))
\end{eqnarray*}
and
\begin{eqnarray*}
\lefteqn{v|_{(T_{\max}/2,T')} \in \bigcap_{s\ge0}C^{\infty}([T_{\max}/2,T'];\mathcal{D}(\underline{\Delta}_{s,p_{2},\gamma}^{2}))\hookrightarrow C^{\infty}([T_{\max}/2,T'];\mathcal{D}(\underline{\Delta}_{s_{1},p_{1},\widetilde{\gamma}}^{2}))}\\
 && \hookrightarrow W^{1,q_{1}}(T_{\max}/2,T';\mathcal{H}_{p_{1}}^{s_{1},\widetilde{\gamma}}(\mathbb{B}))\cap L^{q_{1}}(T_{\max}/2,T';\mathcal{D}(\underline{\Delta}_{s_{1},p_{1},\widetilde{\gamma}}^{2})).
\end{eqnarray*}

Both solutions $v|_{(T_{\max}/2,T')}$ and $u|_{(T_{\max}/2,T')}$ belong to the space 
$$
W^{1,q_{1}}(T_{\max}/2,T';\mathcal{H}_{p_{1}}^{s_{1},\widetilde{\gamma}}(\mathbb{B}))\cap L^{q_{1}}(T_{\max}/2,T';\mathcal{D}(\underline{\Delta}_{s_{1},p_{1},\widetilde{\gamma}}^{2}))
$$ 
for all $T_{\max}/2<T'<T_{\max}$ and are solutions of \eqref{CH1} with initial condition at $T_{\max}/2$ given by 
$$
u(T_{\max}/2)=v(T_{\max}/2)\in(\mathcal{D}(\underline{\Delta}_{s_{1},p_{1},\widetilde{\gamma}}^{2}),\mathcal{H}_{p_{1}}^{s_{1},\widetilde{\gamma}}(\mathbb{B}))_{\frac{1}{q_{1}},q_{1}}.
$$
By uniqueness of Theorem \ref{thcl}, they must be equal.

Finally, we notice that, $u\in C((T_{\max}/2,\infty);\mathcal{H}_{p_{1}}^{\infty,\gamma+2+\delta_{0}}(\mathbb{B})\oplus\mathbb{C}_{\omega})$,
for some $\delta_{0}>0$ as in \eqref{asympembedd}.
As $\mathcal{H}_{p_{1}}^{\infty,\gamma+2+\delta_{0}}(\mathbb{B})\oplus\mathbb{C}_{\omega}$
is a Banach algebra due to Remark \ref{Bancalg}, we have 
$$
F(u)\in C((T_{\max}/2,\infty);\mathcal{H}_{p_{1}}^{\infty,\gamma+\delta_{0}}(\mathbb{B}))\hookrightarrow C((T_{\max}/2,\infty);\mathcal{H}_{p_{2}}^{s_{2},\gamma}(\mathbb{B})),
$$
where we have used again Lemma \ref{propms} and $F(u)=\underline{\Delta}_{s}(u^{3}-u)$. Therefore 
\begin{eqnarray*}
\lefteqn{ \Vert F(v)\Vert _{L^{q_{2}}(0,T_{\max};\mathcal{H}_{p_{2}}^{s_{2},\gamma}(\mathbb{B}))}}\\
 & \le & \Vert F(v)\Vert _{L^{q_{2}}(0,T_{\max}/2;\mathcal{H}_{p_{2}}^{s_{2},\gamma}(\mathbb{B}))}+\Vert F(u)\Vert _{L^{q_{2}}(T_{\max}/2,T_{\max};\mathcal{H}_{p_{2}}^{s_{2},\gamma}(\mathbb{B}))}<\infty.
\end{eqnarray*}
 Hence $[0,T_{\max})$, $T_{\max}<\infty$, cannot be the maximal interval of solution due to Corollary \ref{maximalinterval}. We conclude that $T_{\max}$ is necessarily infinity.
\end{proof}

\subsection{Long time solution}

We are now in position to study the existence of long time solutions of \eqref{CH1}-\eqref{CH2}. The proof of the existence of global solutions is limited to the case $\dim(\mathbb{B})\in\{ 2,3\}$ (see the comment after Proposition \ref{gradest} for further explanation). We proceed as follows:
\begin{enumerate}
\item First we show that for $p=2$ and $q\ge 2$ the solution is bounded
a priori in the space $\mathcal{H}^{1,1+\gamma}_{2}(\mathbb{B})$,
using the energy functional for the Cahn-Hilliard equation.
\item Next, the moment inequality, the singular Gr\"onwall inequality and the use of interpolation theory are employed to show that the solution is actually
bounded in $\mathcal{D}\big((c-\underline{\Delta}_{0})^{\frac{3+\nu}{2}}\big)$, when
$p=2$, $q$ is large and $\nu\in(0,1)$ is properly chosen.
\item Finally, by using \eqref{uexists}, \eqref{extrareg}, \eqref{contsolu} and Corollary \ref{indmaxint}, we obtain the correct estimate to guarantee global existence of solutions
when $\dim(\mathbb{B})=2$ or $3$. 
\end{enumerate}

A similar strategy to prove existence of global solutions can be used to study the Cahn-Hilliard equation on smooth bounded domains in the $L^2$-setting, see e.g. the lecture notes of Larrson \cite{Larsson}. The presence of conical singularities and the use of $L^p$-spaces turn this accomplishment much more technical. Here our regularity results \eqref{uexists}, \eqref{extrareg}, \eqref{contsolu} and Corollary \ref{indmaxint}, are crucial, not only for a very refined understanding of the regularity of the solutions, but also for proving the existence for long times.

\subsubsection{A priori estimate in $\mathcal{H}^{1,1+\gamma}_{2}(\mathbb{B})$}

For the first step, we use the energy functional. We start with the following version of the well-known Green's identity.

\begin{lemma}[Green's identity]\label{green} Assume that $w$ and $v$ belong to $\mathcal{H}_{2}^{1,1}(\mathbb{B})\oplus\mathbb{C}_{\omega}$ and let $\Delta v\in\mathcal{H}_{2}^{0,\gamma}(\mathbb{B})$, for some $\gamma>-1$. Then 
\begin{equation}
\int_{\mathbb{B}}\langle\nabla w,\nabla v\rangle_{g}d\mu_{g}=-\int_{\mathbb{B}}w\Delta vd\mu_{g},\label{eq:Green}
\end{equation}
where $\langle\cdot,\cdot\rangle_{g}$ and $d\mu_{g}$ denote, respectively, the Riemannian scalar product and the Riemannian measure with respect to the metric $g$. The gradient associated to the metric $g$ is denoted by $\nabla$.
\end{lemma}
We note that if $n\ge2$, then $\mathbb{C}_{\omega}\hookrightarrow\mathcal{H}_{2}^{1,1}(\mathbb{B})$.
Therefore, for $\dim(\mathbb{B})\ge3$, we have $\mathcal{H}_{2}^{1,1}(\mathbb{B})\oplus\mathbb{C}_{\omega}=\mathcal{H}_{2}^{1,1}(\mathbb{B})$.
\begin{proof}
Let $\alpha,\beta\ge0$ and let 
$$
\langle\cdot,\cdot\rangle_{\mathcal{H}_{2}^{\alpha,\beta}(\mathbb{B})\times\mathcal{H}_{2}^{-\alpha,-\beta}(\mathbb{B})}:\mathcal{H}_{2}^{\alpha,\beta}(\mathbb{B})\times\mathcal{H}_{2}^{-\alpha,-\beta}(\mathbb{B})\to\mathbb{C}
$$
be the duality map as in Lemma \ref{propms} (ii).
	
Due to the fact that $C_{c}^{\infty}(\mathbb{B}^{\circ})$ is dense in $\mathcal{H}_{2}^{\widetilde{s},\widetilde{\gamma}}(\mathbb{B})$ for all $\widetilde{s},\widetilde{\gamma}\in\mathbb{R}$ and the inclusion $\mathcal{H}_{2}^{0,-\beta}(\mathbb{B})\hookrightarrow\mathcal{H}_{2}^{-\alpha,-\beta}(\mathbb{B})$ is continuous, the duality map satisfies 
\begin{equation}
\langle w,v\rangle_{\mathcal{H}_{2}^{\alpha,\beta}(\mathbb{B})\times\mathcal{H}_{2}^{-\alpha,-\beta}(\mathbb{B})}=\int_{\mathbb{B}}wv\,d\mu_{g},\quad w\in\mathcal{H}_{2}^{\alpha,\beta}(\mathbb{B}), v\in\mathcal{H}_{2}^{0,-\beta}(\mathbb{B}).\label{eq:dual}
\end{equation}
	
This can be easily proved by considering sequences $\left\{ w_{k}\right\} _{k\in\mathbb{N}}$ and $\left\{ v_{k}\right\} _{k\in\mathbb{N}}$ in $C_{c}^{\infty}(\mathbb{B}^{\circ})$ that converge to $w$ in $\mathcal{H}_{2}^{\alpha,\beta}(\mathbb{B})$ and to $v$ in $\mathcal{H}_{2}^{0,-\beta}(\mathbb{B})$, respectively.
	
If $w$ and $v$ belong to $C_{c}^{\infty}(\mathbb{B}^{\circ})$, then the divergence theorem implies that 
$$
\int_{\mathbb{B}}\langle\nabla w,\nabla v\rangle_{g}d\mu_{g}=-\int_{\mathbb{B}}w\Delta v\,d\mu_{g}=-\langle w,\Delta v\rangle_{\mathcal{H}_{2}^{1,1}(\mathbb{B})\times\mathcal{H}_{2}^{-1,-1}(\mathbb{B})}.
$$
	
Therefore if $w$ and $v$ belong to $\mathcal{H}_{2}^{1,1}(\mathbb{B})$, then \eqref{eq:Green} follows easily from the fact that $C_{c}^{\infty}(\mathbb{B}^{\circ})$ is dense in $\mathcal{H}_{2}^{1,1}(\mathbb{B})$, $\Delta:\mathcal{H}_{2}^{1,1}(\mathbb{B})\to\mathcal{H}_{2}^{-1,-1}(\mathbb{B})$ is continuous, $\Delta v\in\mathcal{H}_{2}^{0,\gamma}(\mathbb{B})\subset\mathcal{H}_{2}^{0,-1}(\mathbb{B})$ and \eqref{eq:dual} holds for $\alpha=\beta=1$.
	
When $n=1$, the proof is more complicated. In fact, in this case, $\mathbb{C}_{\omega}\hookrightarrow\cap_{\varepsilon>0}\mathcal{H}_{2}^{\infty,1-\varepsilon}(\mathbb{B})$, but, in general, $u\in\mathbb{C}_{\omega}$ does not belong to $\mathcal{H}_{2}^{s,1}(\mathbb{B})$, for any $s\in\mathbb{R}$.
	
We first note that $\{\cdot,\cdot\}:C_{c}^{\infty}(\mathbb{B}^{\circ})\times C_{c}^{\infty}(\mathbb{B}^{\circ})\to\mathbb{C}$ given by 
$$
(u,v)\mapsto\{u,v\}=\int_{\mathbb{B}}\left\langle \nabla u,\nabla v\right\rangle _{g}d\mu_{g}
$$
has a unique extension to a continuous bilinear map $\{\cdot,\cdot\}:\mathcal{H}_{2}^{1,1+\beta}(\mathbb{B})\times\mathcal{H}_{2}^{1,1-\beta}(\mathbb{B})\to\mathbb{C}$, for all $\beta\ge0$; we use also the notation $\{u,v\}\equiv\left\langle \nabla u,\nabla v\right\rangle _{\mathcal{H}_{2}^{1,1+\beta}(\mathbb{B})\times\mathcal{H}_{2}^{1,1-\beta}(\mathbb{B})}$. This follows from Lemma \ref{propms} (ii), \eqref{mellinsobolevinteger} and the fact that, in local coordinates on the collar neighborhood, $\left\langle \nabla u,\nabla v\right\rangle _{g}$ is given by 
\begin{equation}
\frac{1}{x^{2}}\Big((x\partial_{x}u)(x\partial_{x}v)+\sum_{i,j=1}^{n}h^{ij}(x,y)(\partial_{y_{i}}u)(\partial_{y_{j}}v)\Big).\label{ipn}
\end{equation}
	
Let $w$ and $v$ belong to $\mathcal{H}_{2}^{1,1}(\mathbb{B})\oplus\mathbb{C}_{\omega}$ and $\Delta v\in\mathcal{H}_{2}^{0,\gamma}(\mathbb{B})$ for some $\gamma>-1$. Without loss of generality, we suppose that $\gamma<1$. Then we have that $w=w_{1}+a$ and $v=v_{1}+b$, where $v_{1}$ and $w_{1}$ belong to $\mathcal{H}_{2}^{1,1}(\mathbb{B})$, $\Delta v_{1}\in\mathcal{H}_{2}^{0,\gamma}(\mathbb{B})$ and $a$, $b\in\mathbb{C}_{\omega}$. Therefore 
\begin{eqnarray*}
\lefteqn{\int_{\mathbb{B}}\langle\nabla w,\nabla v\rangle_{g}d\mu_{g}}\\
& \overset{(1)}{=} & -\int_{\mathbb{B}}w_{1}\Delta v_{1}\,d\mu_{g}+\int_{\mathbb{B}}\langle\nabla a,\nabla v_{1}\rangle_{g}d\mu_{g}+\int_{\mathbb{B}}\langle\nabla w_{1},\nabla b\rangle_{g}d\mu_{g}+\int_{\mathbb{B}}\langle\nabla a,\nabla b\rangle_{g}d\mu_{g}\\
& \overset{(2)}{=} & -\int_{\mathbb{B}}w\Delta v\,d\mu_{g}+\int_{\mathbb{B}}a\Delta v_{1}\,d\mu_{g}+\int_{\mathbb{B}}\langle\nabla a,\nabla v_{1}\rangle_{g}d\mu_{g}\\
&& +\int_{\mathbb{B}}w_{1}\Delta b\,d\mu_{g}+\int_{\mathbb{B}}\langle\nabla w_{1},\nabla b\rangle_{g}d\mu_{g}+\int_{\mathbb{B}}a\Delta b\,d\mu_{g}+\int_{\mathbb{B}}\langle\nabla a,\nabla b\rangle_{g}d\mu_{g}.
\end{eqnarray*}
	
The equality $(2)$ is a direct consequence of the definitions of $w$, $v$, $w_{1}$, $v_{1}$, $a$ and $b$. In $(1)$, we used our previous result, as $v_{1}$ and $w_{1}$ belongs to $\mathcal{H}_{2}^{1,1}(\mathbb{B})$.
	
Let us first show that 
$$
\int_{\mathbb{B}}a\Delta v_{1}\,d\mu_{g}+\int_{\mathbb{B}}\langle\nabla a,\nabla v_{1}\rangle_{g}d\mu_{g}=0.
$$
We know that $v_{1}\in\mathcal{H}_{2}^{1,1}(\mathbb{B})\hookrightarrow\mathcal{H}_{2}^{0,\gamma}(\mathbb{B})$ and that $\Delta v_{1}\in\mathcal{H}_{2}^{0,\gamma}(\mathbb{B})$. Therefore $v_{1}\in\mathcal{D}(\underline{\Delta}_{0,2,\gamma,\max})$, where $\underline{\Delta}_{0,2,\gamma,\max}$ is the maximal realization of the Laplacian in $\mathcal{H}_{2}^{0,\gamma}(\mathbb{B})$. By the discussion at the beginning of Section \ref{realizations}, we have that 
$$
\mathcal{D}(\underline{\Delta}_{0,2,\gamma,\max})=\mathcal{D}(\underline{\Delta}_{0,2,\gamma,\min})\oplus\mathcal{E}_{\Delta,\gamma},
$$
where $\mathcal{D}(\underline{\Delta}_{0,2,\gamma,\min})\subset\cap_{\varepsilon>0}\mathcal{H}_{2}^{2,\gamma+2-\varepsilon}(\mathbb{B})$ and $\mathcal{E}_{\Delta,\gamma}$ is a finite dimensional space consisting of functions of the form $\omega(x)c(y)x^{-\rho}\log^{k}(x)$, with $\rho\in\{z\in\mathbb{C}\,|\,\text{Re}(z)\in[-1-\gamma,1-\gamma)\}$ and $k\in\{0,1\}$. As $v_{1}\in\mathcal{H}_{2}^{1,1}(\mathbb{B})\cap\mathcal{D}(\underline{\Delta}_{0,2,\gamma,\max})$, we conclude that there exists an $0<\varepsilon_{0}<1$ such that $v_{1}\in\mathcal{H}_{2}^{2,1+\varepsilon_{0}}(\mathbb{B})$.
	
Now let $\{v_{1,k}\}_{k\in\mathbb{N}}\in C_{c}^{\infty}(\mathbb{B}^{\circ})$ be a sequence such that $v_{1,k}\to v_{1}$ in $\mathcal{H}_{2}^{2,1+\varepsilon_{0}}(\mathbb{B})$. Then, as $\Delta:\mathcal{H}_{2}^{2,1+\varepsilon_{0}}(\mathbb{B})\to\mathcal{H}_{2}^{0,-1+\varepsilon_{0}}(\mathbb{B})$ is continuous, we have 
\begin{eqnarray*}
\lefteqn{\int_{\mathbb{B}}a\Delta v_{1}\,d\mu_{g}=\langle a,\Delta v_{1}\rangle_{\mathcal{H}_{2}^{0,1-\varepsilon_{0}}(\mathbb{B})\times\mathcal{H}_{2}^{0,-1+\varepsilon_{0}}(\mathbb{B})}=\lim_{k\to\infty}\langle a,\Delta v_{1,k}\rangle_{\mathcal{H}_{2}^{0,1-\varepsilon_{0}}(\mathbb{B})\times\mathcal{H}_{2}^{0,-1+\varepsilon_{0}}(\mathbb{B})}}\\
& = & \lim_{k\to\infty}\int_{\mathbb{B}}a\Delta v_{1,k}d\mu_{g}\overset{(3)}{=}-\lim_{k\to\infty}\int_{\mathbb{B}}\langle\nabla a,\nabla v_{1,k}\rangle_{g}d\mu_{g}\\
& = & -\lim_{k\to\infty}\langle\nabla v_{1,k},\nabla a\rangle_{\mathcal{H}_{2}^{1,1+\varepsilon_{0}}(\mathbb{B})\times\mathcal{H}_{2}^{1,1-\varepsilon_{0}}(\mathbb{B})}=-\langle\nabla v_{1},\nabla a\rangle_{\mathcal{H}_{2}^{1,1+\varepsilon_{0}}(\mathbb{B})\times\mathcal{H}_{2}^{1,1-\varepsilon_{0}}(\mathbb{B})},
\end{eqnarray*}
where we used the divergence theorem in $(3)$ and \eqref{eq:dual} several times.
	
To prove that 
$$
\int_{\mathbb{B}}w_{1}\Delta b\,d\mu_{g}+\int_{\mathbb{B}}\langle\nabla w_{1},\nabla b\rangle_{g}d\mu_{g}=0, 
$$
we consider a sequence $\{w_{1,k}\}_{k\in\mathbb{N}}\in C_{c}^{\infty}(\mathbb{B}^{\circ})$ such that $w_{1,k}\to w_{1}$ in $\mathcal{H}_{2}^{1,1}(\mathbb{B})$ while for the proof that $$
\ensuremath{\int_{\mathbb{B}}a\Delta b\,d\mu_{g}+\int_{\mathbb{B}}\langle\nabla a,\nabla b\rangle_{g}d\mu_{g}=0}, 
$$
we use a sequence $\{b_{k}\}_{k\in\mathbb{N}}\in C_{c}^{\infty}(\mathbb{B}^{\circ})$ such that $b_{k}\to b$ in $\mathcal{H}_{2}^{2,1-\varepsilon_{0}}(\mathbb{B})$. The arguments are then very similar as before.
\end{proof}

\begin{proposition}[Gradient estimate]\label{gradest} Let $s=0$,
$p=2$, $\dim(\mathbb{B})\in\{ 2,3\} $, $q\ge2$
and $\gamma$ be chosen as in \eqref{gamma} with the additional assumption: $\gamma<-\frac{1}{2}$ if $\dim(\mathbb{B})=2$ and $\gamma<-\frac{1}{4}$ if $\dim(\mathbb{B})=3$. Then, the solution $u$ of \eqref{CH1}-\eqref{CH2}
given by \eqref{uexists}, \eqref{extrareg}, \eqref{contsolu} on $[0,T]\times\mathbb{B}$
belongs to $C([0,T];\mathcal{H}_{2}^{1,1+\gamma}(\mathbb{B}))$
and satisfies
\begin{gather}\label{gradbound333}
\|u\|_{C([0,T];\mathcal{H}_{2}^{1,1+\gamma}(\mathbb{B}))}\leq C
\end{gather}
for a suitable constant $C$ depending only on $\|\sqrt{\langle\nabla u_{0},\nabla u_{0}\rangle_{g}}\|_{\mathcal{H}_{2}^{0,0}(\mathbb{B})}$
and $\|u_{0}^{2}-1\|_{\mathcal{H}_{2}^{0,0}(\mathbb{B})}$.
\end{proposition}

We note that \eqref{uexists}, \eqref{extrareg}, \eqref{contsolu} require that $s+2>\frac{n+1}{p}$. In Proposition \ref{gradest}, we have $s=0$ and $p=2$, therefore we need $n<3$. This is why we ask for $\dim(\mathbb{B})\in\{2,3\}$. Finally we note also that our choice of $\gamma$ is always possible, according to \eqref{gamma}.

\begin{proof}
By \eqref{extrareg}-\eqref{contsolu} we know that 
$$
u\in C([0,T];\mathcal{H}_{2}^{2,\gamma+2}(\mathbb{B})\oplus\mathbb{C}_{\omega})\cap C^{\infty}((0,T);\mathcal{H}_{2}^{2,\gamma+2}(\mathbb{B})\oplus\mathbb{C}_{\omega}).
$$ 
As $\mathcal{H}_{2}^{2,\gamma+2}(\mathbb{B})\oplus\mathbb{C}_{\omega}\hookrightarrow\mathcal{H}_{2}^{1,1+\gamma}(\mathbb{B})$, we conclude that $u\in C([0,T];\mathcal{H}_{2}^{1,1+\gamma}(\mathbb{B}))$.

In order to prove the boundedness of the solution in $\mathcal{H}_{2}^{1,1+\gamma}(\mathbb{B})$,
we define the following energy functional $\Phi:\mathcal{H}_{2}^{2,\gamma+2}(\mathbb{B})\oplus\mathbb{C}_{\omega}\to\mathbb{R}$
given by
$$
\Phi(u)=\frac{1}{2}\int_{\mathbb{B}}\langle\nabla u,\nabla u\rangle_{g}d\mu_{g}+\frac{1}{4}\int_{\mathbb{B}}(u^{2}-1)^{2}d\mu_{g}.
$$

The functional is well defined. In fact, $\langle\nabla u,\nabla u\rangle_{g}$ is integrable. Moreover, as $p=2$ and $\dim(\mathbb{B})\in\{ 2,3\} $,
that is, $n=1$ or $2$, Remark \ref{Bancalg} tells
us that $\mathcal{H}_{2}^{2,\gamma+2}(\mathbb{B})\oplus\mathbb{C}_{\omega}$
is a Banach algebra. Therefore 
$$
u^{2}-1\in\mathcal{H}_{2}^{2,\gamma+2}(\mathbb{B})\oplus\mathbb{C}_{\omega}\hookrightarrow\mathcal{H}_{2}^{0,0}(\mathbb{B}).
$$

As $\Phi\in C^{1}(\mathcal{H}_{2}^{2,\gamma+2}(\mathbb{B})\oplus\mathbb{C}_{\omega};\mathbb{R})$, the function $[0,T]\ni t\mapsto\Phi(u(t))\in\mathbb{R}$ is continuous and $(0,T)\ni t\mapsto\Phi(u(t))\in\mathbb{R}$ is $C^{1}$.

For $t\in(0,T)$, we have
\begin{eqnarray*}
\partial_{t}\Phi(u(t))&=&\int_{\mathbb{B}}\langle\nabla u'(t),\nabla u(t)\rangle_{g}d\mu_{g}+\int_{\mathbb{B}}u(t)u'(t)(u^{2}(t)-1)d\mu_{g}\\
&=&\int_{\mathbb{B}}u'(t)(-\Delta u(t)+u^{3}(t)-u(t))d\mu_{g},
\end{eqnarray*}
where we used Lemma \ref{green} in the last equality, as $u$ and $u'$ belong to $\mathcal{H}_{2}^{2,\gamma+2}(\mathbb{B})\oplus\mathbb{C}_{\omega}\hookrightarrow\mathcal{H}_{2}^{1,1}(\mathbb{B})\oplus\mathbb{C}_{\omega}$ 
and $\Delta u\in\mathcal{H}_{2}^{2,\gamma+2}(\mathbb{B})\oplus\mathbb{C}_{\omega}\hookrightarrow\mathcal{H}_{2}^{0,0}(\mathbb{B})$, due to the fact that $u\in\mathcal{D}(\underline{\Delta}_{0,2,\gamma}^{2})$.

Denote $J(u)=-\Delta u+u^{3}-u$. For $t\in(0,T)$, $u(t)$ and $u'(t)\in\mathcal{D}(\underline{\Delta}_{0,2,\gamma}^{2})$, due to \eqref{extrareg}. Hence $J(u)\in\mathcal{H}_{2}^{2,2+\gamma}(\mathbb{B})\oplus\mathbb{C}_{\omega}$ and we can use again Lemma \ref{green} and \eqref{CH1} to obtain
$$
\int_{\mathbb{B}}u'J(u)d\mu_{g}=\int_{\mathbb{B}}\Delta J(u)J(u)d\mu_{g}=-\int_{\mathbb{B}}\langle\nabla J(u),\nabla J(u)\rangle_{g}d\mu_{g}\le0.
$$

Therefore, $\partial_{t}\Phi(u)\le0$ and 
\begin{gather}
\Phi(u(t))\leq\Phi(u_{0}),\quad t\in[0,T].\label{Fbound}
\end{gather}

We conclude that 
\begin{equation}\label{estimatenablau}
\int_{\mathbb{B}}\langle\nabla u,\nabla u\rangle_{g}d\mu_{g}\le2\Phi(u_{0})
\end{equation}
and
\begin{eqnarray}\nonumber
\int_{\mathbb{B}}u^{2}d\mu_{g}&=&\int_{\mathbb{B}}(u^{2}-1)d\mu_{g}+\int_{\mathbb{B}}d\mu_{g}\\\label{estimateu2}
&\le&\bigg(\int_{\mathbb{B}}d\mu_{g}\bigg)^{\frac{1}{2}}\bigg(\int_{\mathbb{B}}(u^{2}-1)^{2}d\mu_{g}\bigg)^{\frac{1}{2}}+\int_{\mathbb{B}}d\mu_{g} \le\bigg(\int_{\mathbb{B}}d\mu_{g}\bigg)^{\frac{1}{2}}(4\Phi(u_{0}))^{\frac{1}{2}}+\int_{\mathbb{B}}d\mu_{g}.
\end{eqnarray}

Finally, we show that the integral of $u^{4}$ is also uniformly bounded. In fact
\begin{eqnarray}\label{eq:Estimateintu4}
\int_{\mathbb{B}}u^{4}d\mu_{g}&=&\int_{\mathbb{B}}\left(u^{2}-1+1\right)^{2}d\mu_{g}\le2\int_{\mathbb{B}}\left(\left(u^{2}-1\right)^{2}+1\right)d\mu_{g}\\
&\le&8\Phi\left(u\left(t\right)\right)+2\int_{\mathbb{B}}d\mu_{g}\le8\Phi\left(u_{0}\right)+2\int_{\mathbb{B}}d\mu_{g}.\nonumber
\end{eqnarray}

The estimates \eqref{estimatenablau} and \eqref{estimateu2} give us uniform estimates of the solution in the space $H_{2,\text{loc}}^{1}\left(\mathbb{B}^{\circ}\right)$.
In order to obtain the necessary estimates in a neighborhood of the
conical singularities, we identify it with $\left[0,1\right)\times\partial\mathcal{B}$
and we make the computations working directly with \eqref{mellinsobolevinteger} and \eqref{ipn} for $\mathcal{H}_{2}^{1,1+\gamma}\left(\mathbb{B}\right)$.

For $k+\left|\alpha\right|=1$, we have
\begin{eqnarray}
\lefteqn{\int_{\left[0,1\right)}\int_{\partial\mathcal{B}}\left|x^{\frac{n+1}{2}-(1+\gamma)}\left(x\partial_{x}\right)^{k}\partial_{y}^{\alpha}\left(\omega\left(x\right)u\left(t,x,y\right)\right)\right|^{2}\sqrt{\det\left(h\left(x\right)\right)}\frac{dx}{x}dy}\nonumber\\
&&\le C\int_{\left[0,1\right)}\int_{\partial\mathcal{B}}\frac{1}{x^{2}}\left((x\partial_{x}(\omega u))^{2}+\sum_{i,j=1}^{n}h^{ij}\left(x,y\right)\partial_{y_{i}}\left(\omega u\right)\partial_{y_{j}}\left(\omega u\right)\right)x^{n}\sqrt{\det\left(h\left(x\right)\right)}dxdy\nonumber\\
&&\le C\left(\int_{\mathbb{B}}\left\langle \nabla u,\nabla u\right\rangle _{g}d\mu_{g}+\int_{\mathbb{B}}u^{2}d\mu_{g}\right),\nonumber
\end{eqnarray}
for certain $C>0$. This is true as long as
$$
n+1-2(1+\gamma)-1\ge n-2\iff\gamma\le0,
$$
which is true according to our choice of $\gamma$.

On the other hand, if $k+\left|\alpha\right|=0$, we have
\begin{eqnarray}
\lefteqn{\int_{\left[0,1\right)}\int_{\partial\mathcal{B}}\left|x^{\frac{n+1}{2}-(1+\gamma)}\omega\left(x\right)u\left(t,x,y\right)\right|^{2}\sqrt{\det\left(h\left(x\right)\right)}\frac{dx}{x}dy}\nonumber\\
&&\le\int_{0}^{1}\int_{\partial\mathcal{B}}x^{\frac{n}{2}-2-2\gamma}\left|x^{\frac{n}{4}}u\left(t,x,y\right)\right|^{2}\sqrt{\det\left(h\left(x\right)\right)}dxdy\nonumber\\
&&\le\left(\int_{0}^{1}\int_{\partial\mathcal{B}}\left(x^{\frac{n}{2}-2-2\gamma}\right)^{2}\sqrt{\det\left(h\left(x\right)\right)}dydx\right)^{\frac{1}{2}}\left(\int_{0}^{1}\int_{\partial\mathcal{B}}\left|u\left(t,x,y\right)\right|^{4}x^{n}\sqrt{\det\left(h\left(x\right)\right)}dydx\right)^{\frac{1}{2}}\nonumber\\
&&\le \left(\int_{0}^{1}\int_{\partial\mathcal{B}}x^{n-4-4\gamma}\sqrt{\det\left(h\left(x\right)\right)}dydx\right)^{\frac{1}{2}}\left(\int_{\mathbb{B}}u^{4}d\mu_{g}\right)^{\frac{1}{2}}.\nonumber
\end{eqnarray}

Due to \eqref{eq:Estimateintu4}, we conclude that the above expession
is uniformly bounded in time if, and only if,
$$
n-4-4\gamma>-1\iff\gamma<\frac{n-3}{4}.
$$

For $n=1$ this means that $\gamma<-\frac{1}{2}$ and for $n=2$ that
$\gamma<-\frac{1}{4}$. This coincides with the assumptions and the estimate \eqref{gradbound333} follows.
\end{proof}

It is interesting to note that, due to our use of Mellin-Sobolev spaces, it was necessary to use the estimates of the integral of $u^4$. Usually, in the study of the Cahn-Hilliard equation, only the estimates of $u^2$ and $\langle\nabla u,\nabla u\rangle_{g}$ are used. 

\subsubsection{A priori estimate in $\mathcal{D}\big((c-\underline{\Delta}_{0})^{\frac{3+\nu}{2}}\big)$}

For the second step, we fix a positive constant $c>0$ such that, according to Proposition \ref{bipforbi} and Proposition \ref{biuytpfohit}, the operators defined in \eqref{AsBs} below are invertible sectorial operators which belong to $\mathcal{BIP}(\phi)$ for some $\phi<\frac{\pi}{2}$. Namely, we set
\begin{gather}\label{AsBs}
A_{s}=c-\underline{\Delta}_{s}\quad \text{and}\quad B_{s}=\underline{\Delta}_{s}^{2}+c. 
\end{gather}

In order to obtain the a priori estimate, we will use moment inequalities to estimate the non-linear terms in spaces of different fractional powers. This result follows after the technical lemma below.

\begin{lemma}\label{L56} 
Let $p\in(1,\infty)$, $s\geq0$ and $\gamma$ be as in \eqref{gamma}. Then\\
{\em (i)} For all $\alpha\in(0,1)$, we have
$$
\mathcal{H}_{p}^{s+2\alpha,\gamma+2\alpha}(\mathbb{B})\oplus\mathbb{C}_{\omega}\hookrightarrow\mathcal{D}(A_{s}^{\alpha}).
$$
{\em (ii)} If $\alpha\in(0,1)$ is such that 
\begin{gather}\label{gammarestr}
\gamma+2\alpha-1\notin\Big\{ \pm\sqrt{\Big(\frac{n-1}{2}\Big)^{2}-\lambda_{j}}\,|\,j\in\mathbb{N}\Big\},
\end{gather}
then we have the equality 
$$
\mathcal{D}(A_{s}^{\alpha})=\mathcal{H}_{p}^{s+2\alpha,\gamma+2\alpha}(\mathbb{B})\oplus\mathbb{C}_{\omega}.
$$
 Note that the above sum is direct iff $\gamma+2\alpha\ge\frac{n+1}{2}$.
If $\gamma+2\alpha<\frac{n+1}{2}$, then 
$$
\mathcal{H}_{p}^{s+2\alpha,\gamma+2\alpha}(\mathbb{B})\oplus\mathbb{C}_{\omega}=\mathcal{H}_{p}^{s+2\alpha,\gamma+2\alpha}(\mathbb{B}).
$$
\end{lemma}
\begin{proof}
(i) We proceed as in the proof of \cite[Theorem 3.3]{Ro2}.
By the boundedness of the imaginary powers given by Proposition \ref{bipforbi}, we have, for $\alpha\in(0,1)$,
$$
[X_{0}^{s},X_{1}^{s}]_{\alpha}=\mathcal{D}(A_{s}^{\alpha}),
$$
 up to norm equivalence, see e.g. \cite[(I.2.9.8)]{Am}. By Lemma \ref{prop:interpolation} (iii)
we have that 
\begin{gather}
\mathcal{H}_{p}^{s+2\alpha,\gamma+2\alpha}(\mathbb{B})=[X_{0}^{s},\mathcal{H}_{p}^{s+2,\gamma+2}(\mathbb{B})]_{\alpha}\hookrightarrow[X_{0}^{s},X_{1}^{s}]_{\alpha}.\label{dmm33}
\end{gather}
Moreover, as $\mathbb{C}_{\omega}\hookrightarrow X_{j}^{s}$, for $j\in\{0,1\}$, we have
\begin{equation}\label{dmm34}
\mathbb{C}_{\omega}\hookrightarrow[X_{0}^{s},X_{1}^{s}]_{\alpha}.
\end{equation}
The embeddings \eqref{dmm33} and \eqref{dmm34} imply that
\begin{gather}
\mathcal{H}_{p}^{s+2\alpha,\gamma+2\alpha}(\mathbb{B})\oplus\mathbb{C}_{\omega}\hookrightarrow[X_{0}^{s},X_{1}^{s}]_{\alpha}=\mathcal{D}(A_{s}^{\alpha}).\label{dmm3322}
\end{gather}
(ii) First, we note that 
$$
[X_{0}^{s},X_{1}^{s}]_{\alpha}\subset X_{0}^{s}=\mathcal{H}_{p}^{s,\gamma}(\mathbb{B})\subset\mathcal{H}_{p}^{s+2(\alpha-1),\gamma+2(\alpha-1)}(\mathbb{B}). 
$$
Moreover, due to standard properties of the interpolation, see \cite[Theorem B.2.3]{Haa2}, $\Delta$ maps $[X_{0}^{s},X_{1}^{s}]_{\alpha}$ to 
$$
[\mathcal{H}_{p}^{s-2,\gamma-2}(\mathbb{B}),\mathcal{H}_{p}^{s,\gamma}(\mathbb{B})]_{\alpha}=\mathcal{H}_{p}^{s+2(\alpha-1),\gamma+2(\alpha-1)}(\mathbb{B}),
$$
where we have used again Lemma \ref{prop:interpolation} (iii).
Therefore, $[X_{0}^{s},X_{1}^{s}]_{\alpha}$ belongs to
the maximal domain of $\Delta$ in $\mathcal{H}_{p}^{s+2(\alpha-1),\gamma+2(\alpha-1)}(\mathbb{B})$, i.e. 
\begin{gather}
[X_{0}^{s},X_{1}^{s}]_{\alpha}\hookrightarrow\mathcal{D}(\underline{\Delta}_{s+2(\alpha-1),p,\gamma+2(\alpha-1),\min})\oplus\mathcal{E}_{\Delta,\gamma+2(\alpha-1)}=\mathcal{H}_{p}^{s+2\alpha,\gamma+2\alpha}(\mathbb{B})\oplus\mathcal{E}_{\Delta,\gamma+2(\alpha-1)}.\label{dm22}
\end{gather}
Here we have used \eqref{gammarestr} so that, according to \eqref{Qset} and
the discussion that follows it, 
$$
\mathcal{D}(\underline{\Delta}_{s+2(\alpha-1),p,\gamma+2(\alpha-1),\min})=\mathcal{H}_{p}^{s+2\alpha,\gamma+2\alpha}(\mathbb{B}).
$$
Let $(x,y)=(x,y_{1},...,y_{n})\in(0,1)\times\partial\mathcal{B}$
be local coordinates on the neighbourhood of $\partial\mathcal{B}$. Similarly, the operators
$\omega(x)(x\partial_{x})$ and $\omega(x)\partial_{y_{j}}$,
$j\in\{1,...,n\}$, map $[X_{0}^{s},X_{1}^{s}]_{\alpha}$
to 
\begin{equation}
[\mathcal{H}_{p}^{s-1,\gamma}(\mathbb{B}),\mathcal{H}_{p}^{s+1,\gamma+2}(\mathbb{B})]_{\alpha}=\mathcal{H}_{p}^{s-1+2\alpha,\gamma+2\alpha}(\mathbb{B}).\label{eq:A}
\end{equation}
Moreover, according to \eqref{dm22} we decompose
any $u\in[X_{0}^{s},X_{1}^{s}]_{\alpha}$ as $u=w\oplus v$,
where $v\in\mathcal{E}_{\Delta,\gamma+2(\alpha-1)}$ and
$w\in\mathcal{H}_{p}^{s+2\alpha,\gamma+2\alpha}(\mathbb{B})$.
In the chosen local coordinates we write
\begin{gather}
v=\omega(x)\sum_{i,j}c_{ij}(y)x^{-\rho_{i}}\log^{k_{ij}}(x),\label{sumofv}
\end{gather}
 where the above sum is finite, $c_{ij}\in C^{\infty}(\partial\mathbb{B})$,
$\mathrm{Re}(\rho_{i})\in[\frac{n-3}{2}-\gamma-2(\alpha-1),\frac{n+1}{2}-\gamma-2(\alpha-1))$
and $k_{ij}\in\{ 0,1\} $. 

If $\mathcal{E}_{\Delta,\gamma+2(\alpha-1)}=\emptyset$,
then, by \eqref{dmm3322} and \eqref{dm22}
we deduce
$$
\mathcal{H}_{p}^{s+2\alpha,\gamma+2\alpha}(\mathbb{B})\oplus\mathbb{C}_{\omega}\hookrightarrow[X_{0}^{s},X_{1}^{s}]_{\alpha}\hookrightarrow\mathcal{H}_{p}^{s+2\alpha,\gamma+2\alpha}(\mathbb{B}).
$$
Hence $\gamma+2\alpha<\frac{n+1}{2}$ and 
$$
\mathcal{D}(A_{s}^{\alpha})=[X_{0}^{s},X_{1}^{s}]_{\alpha}=\mathcal{H}_{p}^{s+2\alpha,\gamma+2\alpha}(\mathbb{B}).
$$

Assume now that $\mathcal{E}_{\Delta,\gamma+2(\alpha-1)}\ne\emptyset$,
that is, the set of $\{\rho_{i}\}$ in \eqref{sumofv}
is not empty. Suppose that there exists a $\rho_{i}\neq0$. By applying
$\omega(x)(x\partial_{x})$ on $\omega(x)c_{ij}(y)x^{-\rho_{i}}\log^{k_{ij}}(x)$
and observing that 
$$
\omega(x)\log^{k_{ij}}(x)(x\partial_{x})x^{-\rho_{i}}=-\rho_{i}\omega(x)x^{-\rho_{i}}\log^{k_{ij}}(x)\notin\mathcal{H}_{p}^{s-1+2\alpha,\gamma+2\alpha}(\mathbb{B}),
$$
 we get a contradiction. In fact, as we have seen in \eqref{eq:A}, the above expression should belong to $\mathcal{H}_{p}^{s-1+2\alpha,\gamma+2\alpha}(\mathbb{B})$.

Therefore, the unique possible expression for $v$ in terms of \eqref{sumofv} is $v=\omega(x)(c_{1}(y)\log(x)+c_{2}(y))$ with $c_{1},c_{2}\in C^{\infty}(\partial\mathbb{B})$.
By applying now $\omega(x)\partial_{y_{j}}$, $j\in\{1,...,n\}$,
on $v$, the fact that $\omega(x)\partial_{y_{j}}v\in\mathcal{H}_{p}^{s-1+2\alpha,\gamma+2\alpha}(\mathbb{B})$
implies that either $\gamma<\frac{n+1}{2}-2\alpha$ or $c_{1}=0$
and $c_{2}=constant$. In the first case we get that $\mathcal{E}_{\Delta,\gamma+2(\alpha-1)}\hookrightarrow\mathcal{H}_{p}^{s+2\alpha,\gamma+2\alpha}(\mathbb{B})$
and $\mathbb{C}_{\omega}\hookrightarrow\mathcal{H}_{p}^{s+2\alpha,\gamma+2\alpha}(\mathbb{B})$.
The second case implies $\mathcal{E}_{\Delta,\gamma+2(\alpha-1)}=\mathbb{C}_{\omega}$.
We can summarize the above into
$$
\mathcal{H}_{p}^{s+2\alpha,\gamma+2\alpha}(\mathbb{B})\oplus\mathcal{E}_{\Delta,\gamma+2(\alpha-1)}=\mathcal{H}_{p}^{s+2\alpha,\gamma+2\alpha}(\mathbb{B})\oplus\mathbb{C}_{\omega},
$$
 and the result follows by \eqref{dmm3322} and \eqref{dm22}. 
\end{proof}

\begin{lemma}
\label{firstest} Let
$p\in(1,\infty)$, $s\geq0$, $s+\frac{5}{3}>\frac{n+1}{p}$,
$\gamma$ be as in \eqref{gamma} and $\nu\in(0,1)$
be such that 
$$
\gamma+\frac{5+\nu}{3}-1\notin\Big\{ \pm\sqrt{\Big(\frac{n-1}{2}\Big)^{2}-\lambda_{j}}\,|\, j\in\mathbb{N}\Big\} \quad \text{and} \quad \gamma+\frac{5+\nu}{3}>\frac{n+1}{2}.
$$ 
Then, for each $u\in\mathcal{D}\big(A_{s}^{\frac{3+\nu}{2}}\big)$ we have that 
$$
\|u^{3}-u\|_{\mathcal{D}\big(A_{s}^{\frac{5+\nu}{6}}\big)}\leq C_{0}(1+\|u\|_{\mathcal{D}\big(A_{s}^{\frac{1}{2}}\big)}^{2})\|u\|_{\mathcal{D}\big(A_{s}^{\frac{3+\nu}{2}}\big)},
$$
 with certain $C_{0}>0$ depending only on $A_{s}$ and $\nu$.
\end{lemma}

Note that such a choice of $\nu$ is always possible, as $\{ \pm\sqrt{(\frac{n-1}{2})^{2}-\lambda_{j}}\, |\, j\in\mathbb{N}\} $
is a discrete set and $\gamma+2>\frac{n+1}{2}$, due to \eqref{gamma}.
\begin{proof}
For each $r>0$, denote $\Vert \cdot\Vert _{\mathcal{D}(A_{s}^{r})}$ by $\Vert\cdot\Vert_{r}$. By the moment inequality, see \cite[(V.1.2.12)]{Am},
we have 
\begin{equation}
\Vert u\Vert _{\frac{5+\nu}{6}}\leq C_{1}\Vert u\Vert _{\frac{1}{2}}^{\frac{2}{3}}\Vert u\Vert _{\frac{3+\nu}{2}}^{\frac{1}{3}},\label{eq:Auuu}
\end{equation}
 for each $u\in\mathcal{D}\big(A_{s}^{\frac{3+\nu}{2}}\big)$ and
certain $C_{1}>0$ depending only on $A_{s}$ and $\nu$. By our assumptions
on $\nu$ and Lemma \ref{L56} (ii), we know that
$\mathcal{D}\big(A_{s}^{\frac{5+\nu}{6}}\big)=\mathcal{H}_{p}^{s+\frac{5+\nu}{3},\gamma+\frac{5+\nu}{3}}(\mathbb{B})\oplus\mathbb{C}_{\omega}$.
Hence $\mathcal{D}\big(A_{s}^{\frac{5+\nu}{6}}\big)$ is a Banach algebra up to an equivalent norm, due to Remark
\ref{Bancalg}. Therefore 
\begin{equation}
\Vert u^{3}\Vert _{\frac{5+\nu}{6}}\leq C_{2}\Vert u\Vert _{\frac{5+\nu}{6}}^{3}\label{eq:B}
\end{equation}
 for certain $C_{2}>0$ depending only on $A_{s}$ and $\nu$. Equations
\eqref{eq:Auuu} and \eqref{eq:B} imply that
$$
\Vert u^{3}\Vert _{\frac{5+\nu}{6}}\leq C_{1}^{3}C_{2}\Vert u\Vert _{\frac{1}{2}}^{2}\Vert u\Vert _{\frac{3+\nu}{2}}.
$$
Therefore
$$
\Vert u^{3}-u\Vert _{\frac{5+\nu}{6}} \leq \Vert u^{3}\Vert _{\frac{5+\nu}{6}}+\Vert u\Vert _{\frac{5+\nu}{6}}\leq C_{3}(\Vert u^{3}\Vert _{\frac{5+\nu}{6}}+\Vert u\Vert _{\frac{3+\nu}{2}}) \leq C_{0}(1+\Vert u\Vert _{\frac{1}{2}}^{2})\Vert u\Vert _{\frac{3+\nu}{2}}
$$
for suitable constants $C_{3},C_{0}>0$ depending only on $A_{s}$ and $\nu$.
\end{proof}

We are almost in position to obtain our estimate. We show first the following elementary result.

\begin{lemma}
\label{bipcompare} Let $p\in(1,\infty)$, $s\geq0$,
$\gamma$ be chosen as in \eqref{gamma} and $\underline{\Delta}_{s}$
be the Laplacian \eqref{ddelta}. Then, for any $\sigma\in[0,1]$
we have that 
\begin{gather}
\mathcal{D}(A_{s}^{2\sigma})=\mathcal{D}((A_{s}^{2})^{\sigma})=\mathcal{D}(B_{s}^{\sigma})\label{sfirsteqdom}
\end{gather}
 up to norm equivalence. Moreover, 
\begin{gather}
B_{s}^{-\frac{\sigma}{2}}A_{s}^{\sigma}=A_{s}^{\sigma}B_{s}^{-\frac{\sigma}{2}}\quad\text{in}\quad\mathcal{D}(A_{s}^{\sigma}).\label{seceqdom}
\end{gather}
\end{lemma}

\begin{proof}
Concerning \eqref{sfirsteqdom}, it is sufficient to consider the case $\sigma\in(0,1)$. The first equality in \eqref{sfirsteqdom} follows from Proposition \ref{bipforbi} and the second law of exponents, see \cite[Corollary 3.1.5]{Haa2} or \cite[Lemma 3.6]{RS1}. For the second one, by Proposition \ref{bipforbi}, Proposition \ref{biuytpfohit} and \cite[(I.2.9.8)]{Am} we have that 
\begin{eqnarray*}
\mathcal{D}((A_{s}^{2})^{\sigma})=[\mathcal{D}((A_{s}^{2})^{0}),\mathcal{D}((A_{s}^{2})^{1})]_{\sigma} = [\mathcal{H}_{p}^{s,\gamma}(\mathbb{B}),\mathcal{D}(\underline{\Delta}_{s}^{2})]_{\sigma}=[\mathcal{D}(B_{s}^{0}),\mathcal{D}(B_{s}^{1})]_{\sigma}=\mathcal{D}(B_{s}^{\sigma})
\end{eqnarray*}
 up to equivalent norms. 

Concerning \eqref{seceqdom}, let us first observe that the bounded operators $A_{s}^{-\sigma}$ and $B_{s}^{-\sigma/2}$ commute. In fact, by the Dunford integral formula, it suffices to
show that $A_{s}$ and $B_{s}$ are resolvent commuting (see \cite[(III.4.9.1)]{Am} for definition). For $c>0$ as in \eqref{AsBs}, we have that
$$
B_{s}^{-1}=(\underline{\Delta}_{s}^{2}+c)^{-1}=(\underline{\Delta}_{s}+i\sqrt{c})^{-1}(\underline{\Delta}_{s}-i\sqrt{c})^{-1},
$$
so that $B_{s}^{-1}$ and $A_{s}^{-1}=(c-\underline{\Delta}_{s})^{-1}$ commute. Therefore $A_{s}$ and $B_{s}$ are resolvent commuting and $A_{s}^{-\sigma}$ and $B_{s}^{-\sigma/2}$ commute by \cite[Lemma III.4.9.1 (ii)]{Am}.

The above remarks lead us to the conclusion that 
\begin{gather}\label{eq:C}
A_{s}^{-\sigma}B_{s}^{-\frac{\sigma}{2}}A_{s}^{\sigma}v=B_{s}^{-\frac{\sigma}{2}}v,\quad v\in\mathcal{D}(A_{s}^{\sigma}).
\end{gather}
In particular $B_{s}^{-\sigma/2}v\in\mathcal{D}(A_{s}^{\sigma})$ if $v\in\mathcal{D}(A_{s}^{\sigma})$. Applying $A_{s}^{\sigma}$ to both sides of \eqref{eq:C}, we get \eqref{seceqdom}.
\end{proof}

\begin{proposition}
\label{uuk} Let $\dim(\mathbb{B})\in\{ 2,3\}$, $p=2$, $s=0$, $\gamma<-\frac{1}{2}$ if $\dim(\mathbb{B})=2$, $\gamma<-\frac{1}{4}$ if $\dim(\mathbb{B})=3$ and let $\gamma$ satisfy the conditions of \eqref{gamma}, $\nu$ be as in Lemma \ref{firstest} and $q>\frac{4}{1-\nu}$. Moreover, let $u_{0}$ be as in \eqref{u0reg} and let $u$ be the unique solution of \eqref{CH1}-\eqref{CH2} on $[0,T]\times\mathbb{B}$, for certain $T>0$. Then, we have that $u\in C\big([0,T];\mathcal{D}\big(A_{0}^{\frac{3+\nu}{2}}\big)\big)$ and
$$
\|u(t)\|_{\mathcal{D}\big(A_{0}^{\frac{3+\nu}{2}}\big)}\leq\widetilde{C}\|u_{0}\|_{\mathcal{D}\big(A_{0}^{\frac{3+\nu}{2}}\big)},\quad t\in[0,T],
$$
 for some constant $\widetilde{C}>0$ depending only on $c$, defined
in \eqref{AsBs}, $T$, $\underline{\Delta}_{0}$, $\Vert \sqrt{\langle\nabla u_{0},\nabla u_{0}\rangle_{g}}\Vert _{\mathcal{H}_{2}^{0,0}(\mathbb{B})}$
and $\Vert u_{0}^{2}-1\Vert _{\mathcal{H}_{2}^{0,0}(\mathbb{B})}$. 
\end{proposition}
We note that under the assumptions of Proposition \ref{uuk}, we have $\frac{5}{3}>\frac{n+1}{p}$. Therefore we are also under the conditions of Lemma \ref{firstest}.

\begin{proof}
First we note that, under the above assumptions we have 
\begin{gather}\label{embedinttodom}
(X_{2}^{0},X_{0}^{0})_{\frac{1}{q},q}\hookrightarrow\mathcal{D}\big(A_{0}^{\frac{3+\nu}{2}}\big).
\end{gather}
In fact, by $q>\frac{4}{1-\nu}$, \cite[(I.2.5.4), (I.2.5.2), (I.2.9.8)]{Am} and Lemma \ref{bipcompare} we have that 
$$
(X_{2}^{0},X_{0}^{0})_{\frac{1}{q},q}\hookrightarrow(X_{2}^{0},X_{0}^{0})_{\frac{1-\nu}{4}-\varepsilon,q}\hookrightarrow(X_{0}^{0},X_{2}^{0})_{\frac{3+\nu}{4}+\varepsilon,q}\hookrightarrow[X_{0}^{0},X_{2}^{0}]_{\frac{3+\nu}{4}}=\mathcal{D}\big(B_{0}^{\frac{3+\nu}{4}}\big)=\mathcal{D}\big(A_{0}^{\frac{3+\nu}{2}}\big),
$$
for all sufficiently small $\varepsilon>0$.

By Remark \ref{embedtoC}, we deduce that
$$
u\in C([0,T];(X_{2}^{0},X_{0}^{0})_{\frac{1}{q},q})\hookrightarrow C\big([0,T];\mathcal{D}\big(A_{0}^{\frac{3+\nu}{2}}\big)\big).
$$

Remark \ref{Bancalg} implies then that 
$$
u,u^{3}\in C([0,T];\mathcal{D}(A_{0}))
$$
and therefore 
$$
\Delta(u^{3}-u)\in C([0,T];X_{0}^{0}).
$$

Consider the linear parabolic problem 
\begin{eqnarray}
w'(t)+(\Delta^{2}+c)w(t) & = & e^{-ct}\Delta(u^{3}(t)-u(t)),\quad t\in(0,T),\label{CH3}\\
w(0) & = & u_{0}.\label{CH4}
\end{eqnarray}
 By Theorem \ref{dorevenni}, Theorem \ref{thcl} and Proposition \ref{biuytpfohit},
the above equation has a unique solution 
$$
w\in W^{1,q}(0,T;\mathcal{H}_{2}^{0,\gamma}(\mathbb{B}))\cap L^{q}(0,T;\mathcal{D}(\underline{\Delta}_{0}^{2})).
$$
 Moreover, by \eqref{uexists}, $e^{-ct}u$ is a solution
of \eqref{CH3}-\eqref{CH4} in the
above space. Therefore, $u=e^{ct}w$ and by the variation of constants
formula, see e.g. \cite[Proposition 3.7.22]{ABHN},
we have that 
\begin{gather}
u(t)=e^{ct}e^{-tB_{0}}u_{0}+e^{ct}\int_{0}^{t}e^{(\tau-t)B_{0}}e^{-c\tau}(c-A_{0})(u^{3}(\tau)-u(\tau))d\tau,\quad t\in[0,T].\label{varconst}
\end{gather}

By Lemma \ref{bipcompare} we infer that, for $\sigma\in[0,1]$,
$$
A_{0}^{\sigma}B_{0}^{-\frac{\sigma}{2}},\,B_{0}^{\frac{\sigma}{2}}A_{0}^{-\sigma}\in\mathcal{L}(\mathcal{H}_{2}^{0,\gamma}(\mathbb{B}))\quad\text{and}\quad B_{0}^{\frac{\nu-1}{12}}A_{0}^{\frac{1-\nu}{6}}=A_{0}^{\frac{1-\nu}{6}}B_{0}^{\frac{\nu-1}{12}}\quad\text{in}\quad\mathcal{D}\big(A_{0}^{\frac{1-\nu}{6}}\big)
$$
 Moreover, by \cite[Proposition 2.9]{Ro1} we
have that 
\begin{gather*}
\Vert B_{0}^{\frac{5+\nu}{6}}e^{-tB_{0}}\Vert _{\mathcal{L}(X_{0}^{0})}\leq C_{1}t^{-\frac{5+\nu}{6}},\quad t>0,
\end{gather*}
 for certain constant $C_{1}>0$ depending only on $B_{0}$. Hence,
for $0<\tau<t<T$ by \cite[Lemma III.4.9.1 (iii)]{Am},
Lemma \ref{firstest} and Lemma \ref{bipcompare}
we estimate
\begin{eqnarray}
 \lefteqn{ \Vert A_{0}^{\frac{3+\nu}{2}}e^{(\tau-t)B_{0}}(c-A_{0})(u^{3}(\tau)-u(\tau))\Vert _{X_{0}^{0}}}\nonumber \\
 & = & \Vert A_{0}^{\frac{3+\nu}{2}}B_{0}^{-\frac{3+\nu}{4}}B_{0}^{\frac{5+\nu}{6}}e^{(\tau-t)B_{0}}A_{0}^{\frac{1-\nu}{6}}B_{0}^{\frac{\nu-1}{12}}(c-A_{0})A_{0}^{-1}A_{0}^{\frac{5+\nu}{6}}(u^{3}(\tau)-u(\tau))\Vert _{X_{0}^{0}}\nonumber \\
 & \leq & \Vert A_{0}^{\frac{3+\nu}{2}}B_{0}^{-\frac{3+\nu}{4}}\Vert _{\mathcal{L}(X_{0}^{0})}\Vert B_{0}^{\frac{5+\nu}{6}}e^{(\tau-t)B_{0}}\Vert _{\mathcal{L}(X_{0}^{0})}\nonumber \\
 & & \times\Vert A_{0}^{\frac{1-\nu}{6}}B_{0}^{\frac{\nu-1}{12}}\Vert _{\mathcal{L}(X_{0}^{0})}\Vert (c-A_{0})A_{0}^{-1}\Vert _{\mathcal{L}(X_{0}^{0})}\Vert A_{0}^{\frac{5+\nu}{6}}(u^{3}(\tau)-u(\tau))\Vert _{X_{0}^{0}}\nonumber \\
 & \leq & C_{2}(t-\tau)^{-\frac{5+\nu}{6}}(1+\Vert u\Vert _{\mathcal{D}\big(A_{0}^{\frac{1}{2}}\big)}^{2})\Vert u\Vert _{\mathcal{D}\big(A_{0}^{\frac{3+\nu}{2}}\big)},\label{deuay}
\end{eqnarray}
 where we have used Lemma \ref{firstest} in the last
inequality and $C_{2}>0$ depends only on $c$ and $\underline{\Delta}_{0}$.
We infer, by \cite[Proposition 1.1.7]{ABHN},
that 
$$
\int_{0}^{t}e^{(\tau-t)B_{0}}e^{-c\tau}(c-A_{0})(u^{3}(\tau)-u(\tau))d\tau\in\mathcal{D}\big(A_{0}^{\frac{3+\nu}{2}}\big),\quad t\in[0,T],
$$
and 
\begin{eqnarray}
\lefteqn{A_{0}^{\frac{3+\nu}{2}}\int_{0}^{t}e^{(\tau-t)B_{0}}e^{-c\tau}(c-A_{0})(u^{3}(\tau)-u(\tau))d\tau}\nonumber\\
&=&\int_{0}^{t}A_{0}^{\frac{3+\nu}{2}}e^{(\tau-t)B_{0}}e^{-c\tau}(c-A_{0})(u^{3}(\tau)-u(\tau))d\tau, \quad t\in[0,T].\nonumber
\end{eqnarray}
Moreover, if $t\in[0,T]$, by \eqref{varconst} and \eqref{deuay} we find that 
\begin{eqnarray}
\lefteqn{ \Vert u(t)\Vert _{\mathcal{D}\big(A_{0}^{\frac{3+\nu}{2}}\big)}}\nonumber \\
 & \leq & e^{cT}\Vert A_{0}^{\frac{3+\nu}{2}}B_{0}^{-\frac{3+\nu}{4}}\Vert _{\mathcal{L}(X_{0}^{0})}\Vert e^{-tB_{0}}\Vert _{\mathcal{L}(X_{0}^{0})}\Vert B_{0}^{\frac{3+\nu}{4}}A_{0}^{-\frac{3+\nu}{2}}\Vert _{\mathcal{L}(X_{0}^{0})}\Vert A_{0}^{\frac{3+\nu}{2}}u_{0}\Vert _{X_{0}^{0}}\nonumber \\
 & & +e^{cT}C_{2}\int_{0}^{t}(t-\tau)^{-\frac{5+\nu}{6}}(1+\Vert u\Vert _{\mathcal{D}\big(A_{0}^{\frac{1}{2}}\big)}^{2})\Vert u\Vert _{\mathcal{D}\big(A_{0}^{\frac{3+\nu}{2}}\big)}d\tau\nonumber \\
 & \leq & C_{3}e^{cT}\Big(\Vert u_{0}\Vert _{\mathcal{D}\big(A_{0}^{\frac{3+\nu}{2}}\big)}+\int_{0}^{t}(t-\tau)^{-\frac{5+\nu}{6}}(1+\Vert u\Vert _{\mathcal{D}\big(A_{0}^{\frac{1}{2}}\big)}^{2})\Vert u\Vert _{\mathcal{D}\big(A_{0}^{\frac{3+\nu}{2}}\big)}d\tau\Big)\label{gokserr}
\end{eqnarray}
 for some constant $C_{3}>0$ depending only on $c$ and $\underline{\Delta}_{0}$,
where we have used Lemma \ref{bipcompare} and the
uniform boundedness in $t\in[0,T]$ of $\Vert e^{-tB_{0}}\Vert _{\mathcal{L}(X_{0}^{0})}$, see e.g. \cite[Proposition 2.9]{Ro1}. 

By Lemma \ref{L56} (i) and Proposition \ref{gradest}, there exist some constants $C_{4},C>0$, where $C$ depends only on $\Vert \sqrt{\langle\nabla u_{0},\nabla u_{0}\rangle_{g}}\Vert _{\mathcal{H}_{2}^{0,0}(\mathbb{B})}$ and $\Vert u_{0}^{2}-1\Vert _{\mathcal{H}_{2}^{0,0}(\mathbb{B})}$, such that
\begin{gather}\label{eq:D}
\Vert u(t)\Vert_{\mathcal{D}\big(A_{0}^{\frac{1}{2}}\big)}\le C_{4}\Vert u(t)\Vert_{\mathcal{H}_{2}^{1,\gamma+1}(\mathbb{B})}\le C.
\end{gather}
Here we have used that $\gamma<-\frac{1}{2}$ and $\gamma<-\frac{1}{4}$ for $n=1$ and $n=2$, respectively.
The result now follows by \eqref{gokserr}, \eqref{eq:D}
and singular Gr\"onwall lemma \cite[Lemma 8.1.1]{CH} or \cite[Theorem II.3.3.1]{Am}.
\end{proof}

\subsubsection*{Proof of Theorem \ref{ThLTS}, second part (long time existence)} We now prove that the solutions are globally defined when $\dim(\mathbb{B})\in\{ 2,3\}$, $\gamma<-\frac{1}{2}$ and $\gamma<-\frac{1}{4}$ for $n=1$ and $n=2$, respectively, $\gamma$ satisfies the conditions of \eqref{gamma}, $s\ge0$ and $s+2>\frac{n+1}{p}$, where $p\in(1,\infty)$.

By Corollary \ref{indmaxint} it is enough to prove it for $s=0$, $p=2$, $q>\frac{4}{1-\nu}$, where $\nu$ is as in Lemma \ref{firstest}, and $\gamma$ satisfying the conditions just mentioned. In this situation, we can find a solution $u\in W^{1,q}(0,T;\mathcal{H}_{2}^{0,\gamma}(\mathbb{B}))\cap L^{q}(0,T;\mathcal{D}(\underline{\Delta}_{0}^{2}))$ of the Cahn-Hilliard equation \eqref{CH1}-\eqref{CH2} by \eqref{uexists}, as, for $s=0$, $n\in\{ 1,2\} $ and $p=2$, we have $s+2>\frac{n+1}{p}$.

Suppose that the maximal interval of existence $[0,T_{\max})$ of the solution is finite. Then
\begin{eqnarray*}
 \lefteqn{ \Vert F(u(t))\Vert _{\mathcal{H}_{2}^{0,\gamma}(\mathbb{B})}}\\
 & = & \Vert \underline{\Delta}_{0}(u(t)^{3}-u(t))\Vert _{\mathcal{H}_{2}^{0,\gamma}(\mathbb{B})}\le\Vert u(t)^{3}-u(t)\Vert _{\mathcal{D}(\underline{\Delta}_{0})}\\
 & \le & C_{1}(\Vert u(t)\Vert _{\mathcal{D}(\underline{\Delta}_{0})}^{3}+\Vert u(t)\Vert _{\mathcal{D}(\underline{\Delta}_{0})})\le C_{2}\Big(\Vert u(t)\Vert _{\mathcal{D}\big(A_{0}^{\frac{3+\nu}{2}}\big)}^{3}+\Vert u(t)\Vert _{\mathcal{D}\big(A_{0}^{\frac{3+\nu}{2}}\big)}\Big)\\
 & \le & C_{3}\Big(\Vert u(0)\Vert _{\mathcal{D}\big(A_{0}^{\frac{3+\nu}{2}}\big)}^{3}+\Vert u(0)\Vert _{\mathcal{D}\big(A_{0}^{\frac{3+\nu}{2}}\big)}\Big),
\end{eqnarray*}
for certain $C_{1},C_{2},C_{3}>0$. Here we used the fact that $\mathcal{D}(\underline{\Delta}_{0})$ is a Banach algebra by Remark \ref{Bancalg} and Proposition \ref{uuk} in the last inequality.

We conclude that $\Vert F(u(t))\Vert _{L^{q}(0,T_{\max};\mathcal{H}_{2}^{0,\gamma}(\mathbb{B}))}<\infty$, which is a contradiction to Corollary \ref{maximalinterval}, recalling that $F(\cdot)$ is locally Lipschitz everywhere in $(X_{1}^{0},X_{0}^{0})_{\frac{1}{q},q}$. Hence $T_{\max}=\infty$. \mbox{\ } \hfill $\square$

\begin{remark}
The theorem proved by \cite[Theorem 5.2.1]{PS} actually gives even more information about the regularity of the solutions. In fact, the same theorem implies that
$$
t^{k}\partial_{t}^{k}u\in W^{1,q}(0,T;\mathcal{H}_{p}^{s,\gamma}(\mathbb{B}))\cap L^{q}(0,T;\mathcal{D}(\underline{\Delta}_{s}^{2})), \quad \forall k\in\mathbb{N}. 
$$
Moreover, with precisely the same arguments, we can show that the solution u given by Theorem \ref{ThLTS} belongs to $\bigcap_{s\ge0}C^{\omega}((0,\infty);\mathcal{D}(\underline{\Delta}_{s,p}^{2}))$, where $C^{\omega}$ denotes the set of analytic functions. We just need to note that the function F from Theorem \ref{abstractshortsmooth} also satisfies 
$$
\ensuremath{F|_{(Y_{1}^{k},Y_{0}^{k})_{\frac{1}{q},q}}\in C^{\omega}((Y_{1}^{k},Y_{0}^{k})_{\frac{1}{q},q};Y_{0}^{k})}, \quad k\in\mathbb{N}, 
$$
and repeat exactly the same proofs using \cite[Theorem 5.2.1]{PS} for the analytic case.
\end{remark}

\section{Acknowledgment}
Pedro T. P. Lopes would like to thank Elmar Schrohe for fruitful discussions and the Institut f\"ur Analysis at Leibniz Universit\"at Hannover for the very warm hospitality of the faculty and staff during the visits that allowed him to do this work.

\end{document}